\documentclass{amsart}
\usepackage{amsmath,amssymb,fullpage}
\usepackage{amsthm}
\usepackage{amsfonts,newlfont,url,xspace}
\usepackage{ stmaryrd }
\usepackage[dvipsnames]{xcolor}
\usepackage{graphics,graphicx,verbatim}
\usepackage{float}
\usepackage{hyperref}
\usepackage{soul}
\usepackage[foot]{amsaddr}  
\usepackage{tikz}
\usepackage{mathtools}
\usepackage{verbatim}
\usetikzlibrary{hobby}
\usepackage{cleveref}
\usepackage{relsize}

\newtheorem{thm}{Theorem}[section]
\newtheorem{lem}[thm]{Lemma}
\newtheorem{cor}[thm]{Corollary}
\newtheorem{prop}[thm]{Proposition}
\newtheorem{conj}[thm]{Conjecture}

\theoremstyle{definition}
\newtheorem{definition}[thm]{Definition}
\newtheorem{example}[thm]{Example}
\newtheorem{remark}[thm]{Remark}

\DeclareMathOperator{\lcm}{lcm}

\newcommand\open{\begin{tikzpicture}[scale=0.3]
 \draw (0,0)--(1,1); 
 \draw[thick,blue] (.4,1)--(.4,.4)--(1,.4);
\end{tikzpicture}}
\newcommand\close{\begin{tikzpicture}[scale=0.3]
 \draw (0,0)--(1,1); -
 \draw[thick,blue] (.6,0)--(.6,.6)--(0,.6);
\end{tikzpicture}}
\newcommand\ubounce{\begin{tikzpicture}[scale=0.3]
 \draw (.1,0)--(1,.9);
 \draw[thick,blue] (0,.45)--(.55,.45)--(.55,1);
\end{tikzpicture}}
\newcommand\lbounce{\begin{tikzpicture}[scale=0.3]
 \draw (0,.1)--(.9,1);
 \draw[thick,blue] (.45,0)--(.45,.55)--(1,.55);
\end{tikzpicture}}
\newcommand\fixed{\begin{tikzpicture}[scale=0.3]
 \draw (0,0)--(1,1);
 \fill(.5,.5) circle (.25);
\end{tikzpicture}}

\usepackage[T1]{fontenc}
\usepackage[utf8]{inputenc}
\usepackage{lmodern} 

\newcommand{\circlesign}[1]{
\begin{tikzpicture}[baseline=(X.base), inner sep=1]
    \node[draw,circle] (X)  {\ensuremath{#1}};
\end{tikzpicture}
}

\newcommand{\arrowpattern}{\underset{\mathsmaller{1\shortrightarrow 2}}{\underline{12}}}
\newcommand{\C}{\mathcal{C}} 
\newcommand{\R}{\mathcal{R}} 
\newcommand{\I}{\mathcal{I}} 
\newcommand{\rot}{\textnormal{Rot}} 

\newcommand{\Findstat}{{FindStat}\xspace}

\newcommand{\inv}{\textnormal{inv}}
\newcommand{\Inv}{\textnormal{Inv}}
\newcommand{\maj}{\textnormal{maj}}
\newcommand{\imaj}{\textnormal{imaj}}

\newcommand{\Des}{\textnormal{Des}}
\newcommand{\des}{\textnormal{des}}

\newcommand{\cdes}{\textnormal{cdes}}

\newcommand{\Stat}{\textnormal{Stat}}

\title[Cyclic Sieving on Permutations]{Cyclic sieving on permutations: an analysis of maps and statistics in the FindStat database}
\author[1]{Ashleigh Adams$^1$}
\address[1]{North Dakota State University.
\href{mailto:adams@math.ucdavis.edu}{ashleigh.adams@ndsu.edu}}
\author[2]{Jennifer Elder$^2$}
\address[2]{Missouri Western State University. \href{mailto:jelder8@missouriwestern.edu}{jelder8@missouriwestern.edu}}
\author[3]{Nadia Lafreni\`ere$^3$}
\address[3]{Concordia University. \href{mailto:nadia.lafreniere@concordia.ca}{nadia.lafreniere@concordia.ca}}
\author[4]{Erin McNicholas$^4$}
\address[4]{Willamette University. \href{mailto:emcnicho@willamette.edu}{emcnicho@willamette.edu}}
\author[5]{Jessica Striker$^5$}
\address[5]{North Dakota State University. \href{mailto:jessica.striker@ndsu.edu}{jessica.striker@ndsu.edu}}
\author[6]{Amanda Welch$^6$}
\address[6]{Eastern Illinois University. \href{mailto:arwelch@eiu.edu}{arwelch@eiu.edu}}

\begin{document}

\begin{abstract}
  We perform a systematic study of permutation statistics and bijective maps on permutations using SageMath to search the FindStat combinatorial statistics database to identify apparent instances of the cyclic sieving phenomenon (CSP). Cyclic sieving occurs on a set of objects, a statistic, and a map of order $n$ when the evaluation of the statistic generating function at the $d$th power of the primitive $n$th root of unity equals the number of fixed points under the $d$th power of the map. Of the apparent instances found in our experiment, we prove 34 new instances of the CSP and conjecture three more. 
Our results are organized largely by orbit structure, proving instances of the CSP for involutions with $2^{n-1}$ fixed points and $2^{\lfloor\frac{n}{2}\rfloor}$ fixed points, as well as maps whose orbits all have the same size.  The FindStat maps which exhibit the CSP include a map constructed by Corteel (using a bijection of Foata and Zeilberger) to swap the number of nestings and crossings, the invert Laguerre heap map, a map of Alexandersson and Kebede designed to preserve right-to-left minima, conjugation by the long cycle, as well as reverse, complement, rotation, Lehmer code rotation, and toric promotion.  Our results combined with those of [Elder, Lafreni\`ere, McNicholas, Striker, Welch 2023] show that, contrary to common expectations, actions that exhibit homomesy are not necessarily the best candidates for the CSP, and vice versa.
\end{abstract}

\subjclass{Primary: 05E18; Secondary: 05A05, 05A15}
    
 \keywords{permutations, permutation patterns, Cyclic sieving, dynamical algebraic combinatorics, FindStat}

\maketitle


\section{Introduction}
In \cite{ELMSW22}, we initiated the systematic study of dynamical phenomena on permutations by searching the FindStat database \cite{findstat} for pairs of maps and statistics yielding apparent instances of the \emph{homomesy} phenomenon and proving 122 such instances. In this paper, we continue this study by doing a similar search focusing on another well-studied phenomenon: the cyclic sieving phenomenon (CSP).
Defined by Reiner, Stanton, and White in 2004 \cite{ReStWh2004}, the CSP occurs for a set when evaluating a polynomial at particular roots of unity counts elements fixed under corresponding powers of an action on that set. It has been observed on various objects, including tableaux, words, set partitions, paths, permutations, and Catalan objects (see  \cite{CSPwebsite,WhatIsCSP,SaganCSP}).
One challenge to finding instances of this phenomenon is that it is necessary to compute two distinct quantities and show they coincide: orbit structure of an action and root of unity evaluations of a polynomial. Sometimes the action relates to the polynomial evaluation, but not always, so the choice of action and polynomial often requires inspired guessing or representation theoretic motivation. Our approach is the first systematic search for instances of this phenomenon.

The first example in the original CSP paper
\cite[Theorem 1.1]{ReStWh2004} includes Mahonian permutation statistics under rotation as a special case.
Another instance, proved in \cite{BRS} and also mentioned in \cite{CSPwebsite}, is related to the Robinson-Schensted-Knuth correspondence and involves the action that cyclically rotates both rows and columns of a permutation matrix and uses the generating function for the sum of the squares of the $q$-hook length formula.
The only other occurrences of the cyclic sieving phenomenon on permutations we could find in the literature concern the map of conjugation by the long cycle $(1, 2, \ldots, n)$.
We discuss this prior work in Section~\ref{sec:conj}.

We begin in Section~\ref{sec:methods} by discussing our systematic approach to the CSP on permutations using FindStat. Note that when stating a CSP result, one typically lists a triple consisting of the set, the polynomial, and the map. For every result in this paper, the set considered is the set of permutations of $[n]$ and the polynomial is the generating function of a permutation statistic. Thus, instead of referring to such a triple, we say a given permutation \emph{statistic} exhibits the CSP with respect to the orbit structure of a given map.
We also summarize the results of this paper in light of our prior study of homomesy \cite{ELMSW22} (Remark \ref{rem:list_of_overlap}).   Section~\ref{sec:bg} gives background and definitions related to permutations and the CSP. 
Our main results comprise Sections~\ref{sec:corteel} through \ref{sec:basic}, in which we
organize instances of the CSP on permutations primarily by orbit structure.
The results in Sections~\ref{sec:corteel} through \ref{sec:rev_comp}  concern involutions. If an involution exhibits the CSP for a statistic, the evaluation of the statistic generating function at $-1$ gives the number of fixed points; this is Stembridge's $q=-1$ phenomenon~\cite{JStembridge,Stembridge1994}. 

Section~\ref{sec:corteel} studies involutions with $2^{n-1}$ fixed points. FindStat contains two involutions that we show in Lemmas~\ref{lem: Corteel fixed pts} and \ref{lem: Laguerre fixed pts} have $2^{n-1}$ fixed points: the \textbf{Corteel} and the \textbf{invert Laguerre heap} maps. The Corteel map switches the numbers of crossings and nestings, while the invert Laguerre heap map can be seen as a  reflection of Reading's noncrossing arc diagram that provides a combinatorial model for canonical join representations~\cite{reading2015noncrossing}. 
The proof techniques in this section are the most subtle of this paper
and involve some proofs of statistic equidistribution that may be of independent interest (Lemmas \ref{lem:317_373_1744} and \ref{lem:equid_371}). We prove that several classes of statistics with different generating functions exhibit the CSP with respect to this orbit structure. These include the number of crossings (Theorem~\ref{thm: Corteel-crossings CSP}), the cycle descent number (Theorem \ref{thm:cycle_des_CSP}), 
the number of midpoints of decreasing subsequences of length 3 (Theorem \ref{conj:371}), and the number of occurrences of the pattern $32-1$ (Theorem  \ref{thm:357}).

Section~\ref{sec:alex-k} studies involutions with exactly  $2^{\lfloor \frac{n}{2}\rfloor}$ fixed points. The notable example in FindStat is the \textbf{Alexandersson-Kebede} map, an involution on permutations that preserves right-to-left minima,  introduced in \cite{AlexanderssonKebede} in order to give a bijective proof of an identity regarding derangements. We prove instances of the CSP for this orbit structure and statistics related to extrema of partial permutations (Theorem \ref{thm:CSP_AK} and Corollary \ref{cor:1005}) and to  inversions (Theorem \ref{thm:CSP_AK-Invisible}).

The maps that exhibit the most instances of the CSP on permutations in our study are involutions with no fixed points; prominent examples of these are the \textbf{reverse} and \textbf{complement} maps. 
We prove CSPs for these in Section~\ref{sec:rev_comp} using two main techniques: (1) finding the specific form of the statistic generating function and  using it to show that $f(-1)=0$ (for example, Theorem \ref{circled_shifted}), and (2) constructing specific involutions $\psi:S_n \rightarrow S_n$ to pair the permutations such that $\text{stat}(\sigma)$ and $\text{stat}(\psi(\sigma))$ have the opposite parity (for example, Theorem \ref{conj:patterns}).
Some statistics of interest (with distinct generating function) include:
the number of circled entries of the shifted recording tableau (Theorem \ref{circled_shifted}), 
 the number of inversions of distance at most $2$ (Theorem \ref{inv_at_most_2}),  
and the standardized bi-alternating inversion number (\Cref{thm:bi_alt}).

In Section~\ref{sec:basic}, we complete our study by focusing on specific statistics and proving the CSP for multiple orbit structures at once. 
We establish several CSPs for maps whose orbits all have the same size 
(Theorem \ref{Mahonian} and Proposition
\ref{thm:rank}). Such maps include  \textbf{rotation}, \textbf{Lehmer code rotation} and \textbf{toric promotion}.  Featured statistics include the Mahonian statistics, such as the major index and the number of inversions. We also find that the $i$-th entry of a permutation exhibits the CSP with respect to the rotation map (Theorem \ref{thm:ith_entry}), and that toric promotion has the CSP for the number of inversions of the second entry (Corollary \ref{cor:toric_pro_2nd_entry}) and a descent variant minus the inversion number (Proposition \ref{thm:1911}).

\subsection*{Acknowledgments} We thank Theo Douvropoulos for suggesting to extend the methods of our previous paper \cite{ELMSW22} to the cyclic sieving phenomenon and Bishal Deb, Sergi Elizalde, and the anonymous referees for helpful comments. We thank the developers of SageMath and FindStat for developing computational tools used in this research. This research began during the authors' Jan.\ 2023 visit to the Institute for Computational and Experimental Research in Mathematics in the Collaborate@ICERM program. Much of it was completed at the 2023 Dynamical Algebraic Combinatorics mini-conference at North Dakota State University, with funding from NSF grant DMS-2247089 and the NDSU foundation. JS acknowledges support from Simons Foundation gift MP-TSM-00002802 and NSF grant DMS-2247089.


\section{A systematic, statistic-oriented approach to the CSP}
\label{sec:methods}
Our approach to finding instances of the CSP is via a systematic search, rather than testing pairs of specific maps and polynomials that arise in other research. For this study, we focused our attention on maps and statistics on permutations in \Findstat---an online database of combinatorial maps and statistics developed by Berg and  Stump in 2011 and currently maintained by Rubey and Stump \cite{findstat}.  Like other fingerprint databases highlighted in~\cite{fingerprint}, \Findstat is a searchable collection of information about a class of mathematical objects.  
The interface between SageMath \cite{sage} and \Findstat makes it possible to not only retrieve stored information from the \Findstat database, but use that information to conjecture new connections.

For each statistic on permutations, we used SageMath to evaluate the statistic generating function at the appropriate roots of unity to test whether this matched the corresponding fixed point counts of any bijective map on permutations in the database. The experiment we conducted was based on the \Findstat database as of January 30, 2024, and was done for all permutations of $4$ to $6$ elements. This range in $n$ was chosen to limit false negatives, as some statistics are not meaningful for very small permutations, and to be executed in a reasonable amount of time. 
We looked at all pairs made of the 24 bijective maps and 400 permutation statistics in \Findstat, and found counter-examples to the cyclic sieving phenomenon for all but 200 pairs. After removing repetitions for pairs that are known to have the same orbit structure and the same statistic generating function, we found 36 potential instances of the CSP, meaning distinct pairs made of an orbit structure and a statistic generating function. Four of these  were previously proven CSPs discussed in the introduction.
Of the remaining 32, we prove 29 new instances of the CSP, prove one for  odd values of $n$ and conjecture it is true for even $n$, and leave two others as conjectures.  We also provide three new interesting proofs of equidistribution of statistics.
In total, the number of pairs of a map and a permutation statistic from FindStat for which we have a proof of cyclic sieving for all values of $n \geq 4$ is $194$. 

For all other pairs of a permutation statistic and a map in FindStat, we found counter-examples to the CSP. This means that pairs of maps and statistics in FindStat that are not mentioned here do not exhibit the CSP, or at least not for all values of $n \geq 4$. For instance, the inverse map on permutations did not appear in our search for potential CSPs, so this map does not exhibit the CSP for all $n$ with any statistic in FindStat. Therefore, even though we present here CSPs for several involutions (with a varying number of fixed points), there is no statistic listed in FindStat that exhibit the phenomenon under involutions with as many fixed points as there are standard young tableaux (which is the number of fixed points of the inverse map). While working on the project, we noticed that it is possible that some statistics exhibit the CSP only for an infinite subset of values of $n$. We give here four examples of statistics that have the CSP for all odd or all even values of $n$ (see Corollary \ref{cor:Corteel_1004} and Theorem \ref{strict_3_des}), but we did not attempt a systematic study of the CSP for only odd or even values. Including these results, we prove 34 instances of the CSP in total.

The statistic generating function information in FindStat is largely empricially derived and only stored for $n\leq 6.$ While FindStat will suggest possible maps or compositions of maps that might provide a bijection between statistics, these suggestions are based on data for small $n$ and do not prove equidistribution.  Some of our equidistribution proofs are inspired by FindStat suggestions, but in many cases we construct a novel bijection based on the definitions of the statistics to prove equidistribution. 

In the statements about our results, we refer to statistics and maps by their name as well as their \Findstat identifier (ID).  This makes it easier for the reader to find the associated \Findstat webpage.  For example, the URL for Statistic $356$, the number of occurrences of the pattern $13-2$, is: \url{http://www.findstat.org/StatisticsDatabase/St000356}. The highest ID for a bijective map on permutations was 310, and the highest ID for a statistic on permutations was 1928.

Due to the breath of the FindStat database, our study gives a good portrait of occurrences of the phenomenon on permutations.  In \cite{ELMSW22}, we conducted a similar study of a dynamical phenomenon on permutations called homomesy.  The results of the two projects combined give a good portrait of the differences between homomesy and the CSP for permutation maps and statistics, as noted in the remark below.

\begin{remark}\label{rem:list_of_overlap} In past studies, some objects and actions were found that exhibit both homomesy and the CSP; it was therefore often believed that the best candidates for exhibiting the CSP were known occurrences of homomesy, and vice versa. 
It is therefore interesting to consider how
many of the map--statistic pairs from our homomesy paper \cite{ELMSW22}, as well as the follow-up \cite{dowling2023homomesy}, also exhibit the CSP.  We found the only pairs of maps and statistics that exhibit both homomesy and the CSP are in Sections \ref{sec:rev_comp} and \ref{sec:basic} of this paper. We list these below, organized by map.

\begin{itemize}
    \item Lehmer code rotation: Statistic 20.
    \item Rotation: Statistics 54, 740, 1806, and 1807.
    \item Both reverse and complement: Statistic 18 (Mahonian); Statistics 495, 538, and 677; Statistics 21 and 1520 (for $n$ even); and Statistics 836 and 494 (for $n$ odd).
    \item Complement but not reverse: 
 Statistics 4, 20, and 692 (Mahonian); Statistics 1114 and 1115.
    \item Reverse but not complement: 
 Statistics 304, 305, 446, and 798 (Mahonian).
\end{itemize}

  In this paper, the reverse and complement are paired together as they have the same orbit structure. However, their lists are different in \cite{ELMSW22} because homomesy does not solely rely on orbit structure. In \cite{dowling2023homomesy}, rotation exhibits homomesy when paired with 34 statistics, but the only statistics that also exhibit the CSP are those related to specific entries. We also note that the Lehmer code rotation exhibits homomesy when paired with 45 statistics, and only exhibits the CSP when paired with one statistic.

In \cite{ELMSW22}, we found nine permutation maps that exhibit homomesy, and six of those maps do not exhibit any instances of the CSP. Similarly, in this paper we consider Corteel and invert Laguerre heap, Alexandersson-Kebede, and conjugation by the long cycle, none of which exhibit any homomesy. Finally, we include the map toric promotion in Section \ref{sec:basic} of this paper, but it is an open question whether the map exhibits homomesy with any permutation statistics. We did not investigate the latter, since the toric promotion map was not in FindStat at the time of the investigation of \cite{ELMSW22}.
\end{remark}

It would be interesting to conduct investigations of the CSP or homomesy involving compositions of FindStat maps. We did not attempt this in the current paper or \cite{ELMSW22}, but leave this to future work. Another interesting direction for future study would be to use the approach of \cite{AlexAmini} to find abstract conditions under which a generating function on permutations may exhibit the CSP.


\section{Background}\label{sec:bg}
In this section, we collect background information on permutations (Subsection~\ref{sec:bg_perm}) and the cyclic sieving phenomenon (Subsection~\ref{sec:bg_CSP}), as well as prior work on the CSP on permutations with respect to conjugation by the long cycle (Subsection~\ref{sec:conj}). Readers  may skip familiar subsections or refer to them as needed.
\subsection{Permutations}
\label{sec:bg_perm}
\begin{definition}
  A \textbf{permutation} of $[n]:=\{1,2,\ldots n\}$ is a bijective map $\sigma:[n]\rightarrow[n]$. We often denote a permutation $\sigma$ as an ordered list $\sigma_1\sigma_2\ldots\sigma_n$ where $\sigma_i=\sigma(i)$, and refer to this as the one-line notation of~$\sigma$. The set of permutations of $[n]$ is often denoted as $S_n$.
\end{definition}

We define below a few common maps and statistics on the set of all permutations; less familiar maps and statistics will be defined in the relevant sections.

\begin{definition}\label{def:maps}
  For a permutation $\sigma = \sigma_1 \ldots \sigma_n$,  the \textbf{reverse} of $\sigma$ is $\R(\sigma) = \sigma_n\ldots \sigma_1$, and the \textbf{complement} of $\sigma$ is $\C(\sigma) = (n+1-\sigma_1)\ldots(n+1-\sigma_n)$. That is,   $\R(\sigma)_i=\sigma_{n+1-i}$ and $\C(\sigma)_i =n+1 - \sigma_i$. Furthermore, the \textbf{inverse} of $\sigma$, denoted $\I(\sigma),$ is the permutation  such that $\I(\sigma) \circ \sigma = \sigma \circ \I(\sigma) = 1\cdots n$, and the \textbf{rotation} of $\sigma$ is $\rot(\sigma) = \sigma_2\sigma_3\ldots \sigma_n\sigma_1$.
\end{definition}

\begin{definition}
  \label{def:common_stats}
  We say that $(i,j)$ is an \textbf{inversion} of a permutation $\sigma=\sigma_1\sigma_2\ldots\sigma_n$ if $i<j$ and $\sigma_j < \sigma_i$. We write $\Inv(\sigma)$ for the set of inversions in $\sigma$. We say that $(\sigma_i, \sigma_j)$ is an \textbf{inversion pair} of $\sigma$ if $(i, j)$ is an inversion of $\sigma$. The \textbf{inversion number} of a permutation $\sigma$, denoted $\inv(\sigma)$, is the number of inversions.

  We say that $i$ is a \textbf{descent} whenever $\sigma_i > \sigma_{i+1}$. We write $\Des(\sigma)$ for the set of descents in $\sigma$, and $\des(\sigma)$ for the number of descents.  If $i \in [n-1]$ is not a descent, we say that it is an \textbf{ascent}.
  The \textbf{major index} of $\sigma$, denoted $\maj(\sigma)$, is the sum of its descents.

  We say that $\sigma_i$ is a \textbf{left-to-right-maximum} of $\sigma$ if there does not exist $j<i$
  such that $\sigma_j>\sigma_i$. Left-to-right minima, right-to-left maxima, and right-to-left minima are defined similarly.
\end{definition}

\begin{definition}\label{def:fund_transform}
  Let $\sigma = \sigma_1 \ldots \sigma_n$ be a permutation. We define the \textbf{first fundamental transform}, denoted $\mathcal{F}(\sigma)$, as follows:
  \begin{itemize}
    \item Place parenthesis $($ to the left of $\sigma_1$, and $)$ to the right of $\sigma_n$.
    \item Before each left-to-right maximum in $\sigma$, insert a parenthesis pair $)($.
  \end{itemize}
  The resulting parenthesization is the cycle decomposition for $\mathcal{F}(\sigma)$ \cite{findstat}. For example, for $\sigma = 241365$, we have $\mathcal{F}(\sigma)=(2)(413)(65)$.
\end{definition}

\begin{remark}\label{remark:fund_bij}
  We note that Stanley defines the inverse of $\mathcal{F}$ as the \textbf{fundamental bijection} \cite[Page 23]{Stanley2011}, which is also known as the first fundamental transform.
  Stanley shows that each cycle in $\sigma$ is mapped to a left-to-right maximum in $\mathcal{F}^{-1}(\sigma)$. Thus $\mathcal{F}$ also maps each left-to-right maximum to a cycle in $\mathcal{F}(\sigma)$. 
  For clarity within the context of this paper, we use the definition given on FindStat, even though it uses a convention opposite to that of Stanley.
\end{remark}

We also study statistics related to permutation patterns of length $3$.
\begin{definition}
  \label{def:patterns}
  Let $\pi=abc\in S_3$ and $\sigma\in S_n$. The permutation $\sigma$ \textbf{contains the pattern $abc$} if there exist $1\leq i<j<k\leq n$ such that $\sigma_i,\sigma_j,\sigma_k$ are in the same relative order as $a,b,c$. The permutation $\sigma$ \textbf{contains the consecutive pattern $a-bc$} if, in addition, $k=j+1$, that is, if $\sigma_j$ and $\sigma_k$ are consecutive in the one-line notation of $\sigma$.
\end{definition}

Finally, we define and give an example of a statistic generating function, as these are central objects in this paper.
\begin{definition}
  \label{def:stat_gen_fun}
  Let $n\geq 2$ be an integer, and $\textrm{stat}(\sigma)$ be a permutation statistic defined on all $\sigma \in S_n$. A \textbf{statistic generating function} on $S_n$ is a polynomial defined as follows:
  \[f(q) = \sum_{\sigma \in S_n} q^{\textrm{stat}(\sigma)}.\]
\end{definition}
Note that one immediate consequence of this definition is that $f(1)=n!$.

\begin{example}
  We note that in $S_3$, there is one permutation with 0 descents, four permutations with 1 descent, and one permutation with 2 descents. Therefore,
  \[f(q) = \sum_{\sigma \in S_3} q^{\textrm{des}(\sigma)} = 1+4q^1+q^2.\]
\end{example}

We say that two statistics are \textbf{equidistributed} if they have the same statistic generating function.

\subsection{The cyclic sieving phenomenon}
\label{sec:bg_CSP}
In 2004, the notion of cyclic sieving was first introduced by Reiner, Stanton, and White with the following definition.

\begin{definition}[\cite{ReStWh2004}]
  Given a set $S$, a polynomial $f(q)$, and a bijective action $g$ of order $n$, the triple $(S,f(q),g)$ exhibits the \textbf{cyclic sieving phenomenon} 
  if, for all $d$, $f(\zeta^d)$
  counts the elements of $S$ fixed under $g^d$, where $\zeta=e^{2\pi i /n}$.
\end{definition}

While at first glance it seems unlikely to find such a triple that both works and is of interest, there have been many examples. In \cite{SaganCSP}, Sagan surveyed the current literature to show examples of CSP involving mathematical objects such as Coxeter groups, Young tableaux, and the Catalan numbers. Further, more examples of CSP on mathematical objects such as words, paths, matchings and crossings can be found at the online Symmetric Functions Catalog \cite{CSPwebsite} maintained by Alexandersson. For a brief and insightful introduction to cyclic sieving, we refer readers to \cite{WhatIsCSP}, where more examples of CSP occurring for a variety of mathematical objects can be found.

In a paper pre-dating the  definition of the CSP \cite{JStembridge}, Stembridge presented the ``$q = -1$'' phenomenon that coincides with the CSP when the order of the bijective action is $2$. In that case, showing that the triple exhibits CSP is equivalent to showing that $f(1)$ equals the cardinality of $S$ and $f(-1)$ is the number of fixed points of the bijection. So, if the action $g$ is an involution, then $f(-1)$ gives a closed form enumeration of the fixed points. 
This ``$q = -1$'' phenomenon helped inspire the definition of the CSP and is integral to the work of this paper, as many of the maps we investigate are involutions.

Statistics with the same generating function exhibit the CSP for the same orbit structures. In fact, it suffices for generating functions to be the same modulo $q^{n-1}$. In some cases, we do not know explicitly the statistic generating function, but proved equidistribution through statistic-preserving bijections.

\subsection{Prior work on the CSP for conjugation by the long cycle}
\label{sec:conj}
Conjugation by the long cycle $(1, 2, \ldots, n)$  (FindStat Map 265) has orbits whose sizes are the divisors of $n$. 
Using representation theory, Barcelo, Reiner, and Stanton proved that some pairs of Mahonian statistics exhibit the \emph{bicyclic sieving phenomenon} \cite[Theorem 1.4]{BRS}, which implies linear combinations of these  statistics exhibit the CSP.
Specifically, we have the following corollary of \cite[Theorem 1.4]{BRS}.

\begin{cor}
  The following statistics exhibit the CSP under the conjugation by the long cycle map:
  \begin{itemize}
    \item Statistic $825$: The sum of the major and the inverse major index.
    \item Statistic $1379$: The number of inversions plus the major index.
    \item Statistic $1377$: The major index minus the number of inversions.
    \item The major index minus the inverse major index.
  \end{itemize}
\end{cor}

\begin{proof}
  In the case $\sigma$ equals the identity and $W=S_n$, \cite[Theorem 1.4]{BRS}  implies that the bimahonian statistic pairs $(\inv,\maj)$ and $(\maj,\imaj)$ exhibit the biCSP. As a consequence, the linear combinations $\maj-\inv$, $\maj-\imaj$, $\maj+\inv$, and $\maj+\imaj$ each exhibit the CSP. $\maj-\imaj$ is not in FindStat, but the other three are.
\end{proof}

Later, Sagan, Shareshian and Wachs  in \cite[Theorem 1.2]{SSW} showed that the major index minus the number of excedances (Statistic 462) exhibits the cyclic sieving phenomenon for conjugation by the long cycle.  The proof involved an evaluation of the coefficients of Eulerian polynomials.

\begin{thm}[\protect{\cite[Theorem 1.2]{SSW}}] The major index minus the number of excedances (Statistic $462$) exhibits the CSP under conjugation by the long cycle.
\end{thm}

We collect subsequent results on equidistributed statistics in the following corollary; specific citations are given in the proof.
\begin{cor}
The following statistics are equidistributed with the major index minus the number of excedances, and so exhibit the CSP with respect to conjugation by the long cycle.
\begin{itemize}
    \item Statistic $463$: The number of admissible inversions of a permutation.
    \item Statistic $866$: The number of admissible inversions of a permutation in the sense of Shareshian-Wachs.
    \item Statistic $961$: The shifted major index of a permutation.
  \end{itemize}
  \end{cor}

\begin{proof}
  Both notions of admissible inversions, and were defined in the context of Eulerian polynomials. An \textit{admissible inversion} (Statistic 463) of a permutation $\sigma$ is defined by Lin and Zeng  \cite{LinZeng}  as an inversion $(i,j)$ that further satisfies that $1<i$ and $\sigma(i-1)<\sigma(i)$ or that there is some index $k$ between $i$ and $j$ for which $\sigma(i)<\sigma(k)$. Similarly, Shareshian and Wachs gave the definition of an admissible inversion (Statistic 866) as an inversion $(i,j)$ for which either $\sigma(j)<\sigma(j+1)$ or there exists an index $k$ between $i$ and $j$ satisfying $\sigma(k)<\sigma(j)$ \cite{ShareshianWachs2007}. For a permutation $\sigma$, consider its reverse $\R(\sigma)$ and its complement, $\C(\sigma)$. Then, if $(i,j)$ is an admissible inversion of $\sigma$ in the sense of Lin and Zeng, we claim that $(i,j)$ is an admissible inversion of $\R\circ\C(\sigma)$ in the sense of Shareshian and Wachs.
  One can check that for a permutation $\sigma$, $\sigma(i)>\sigma(j)$ exactly when $\R\circ\C(\sigma)(i)>\R\circ\C(\sigma)(j)$ for any $(i, j)$. Then, applying the reverse and the complement to a permutation transforms the admissible inversions in the sense of Lin and Zeng to admissible inversions in the sense of Shareshian and Wachs.

  Shareshian and Wachs also showed \cite[Theorem 4.1]{ShareshianWachs2007} that the major index minus the number of excedances is equidistributed with the number of admissible inversions (in their sense).

  Finally, Shareshian and Wachs showed \cite[Theorem 9.7]{ShareshianWachs2016} that the major index minus the number of excedances is equidistributed with the shifted major index, which is defined as the sum
$\displaystyle\sum_{\substack{i \in [n-1]\\ \sigma(i)>\sigma(i+1)+1}} i$.
\end{proof}


\section{Involutions with $2^{n-1}$ fixed points}
\label{sec:corteel}
In this section, we prove CSPs on involutions with exactly $2^{n-1}$ fixed points.  We give examples of two maps having this orbit structure in Subsection \ref{subsec:CorteelLaguerre}, and prove CSP results for four classes of equidistributed statistics in Subsections \ref{subsec: CorteelCSP}-\ref{subsec:Corteel 321}.

\subsection{Corteel and invert Laguerre heap maps}\label{subsec:CorteelLaguerre}


We begin by defining two involutions and showing they each have $2^{n-1}$ fixed points.  
The first map, which FindStat calls {Corteel's map}, was constructed in \cite{corteel2007crossings} to interchange the number of crossings and the number of nestings of a permutation, thus giving  a combinatorial proof of their equidistribution.  It applies a bijection of Foata and Zeilberger~\cite{FoataZeilberger} to a permutation, then takes the complement of the resulting colored Motzkin path relative to the height of the path, and then applies the inverse Foata-Zeilberger bijection to find the corresponding permutation. (Similar bijections between labelled Motzkin paths and permutations are due to Francon and Viennot~\cite{FranconViennot1979}.)  In Definition~\ref{def: Corteel} we explain each step more precisely, giving visual examples that rely heavily on the cycle diagrams of Elizalde found in \cite{elizalde2018continued}. See also \cite[\S6]{BishalDeb} for further discussion and examples.

\begin{definition}\label{def: Corteel}
  For a given permutation $\sigma$, vertical lines on the cycle diagram connect the point $(i,i)$ to $(i,\sigma(i))$, horizontal lines connect $(i,i)$ to $(\sigma^{-1}(i),i)$ (see Figure \ref{fig: CycleDiagramEx}).  Each vertex $(i,i)$ in the cycle diagram of $\sigma\in S_n$ can be classified into one of five types: an upward facing bracket \open, a downward facing bracket \close, a bounce from below \lbounce, a bounce from above \ubounce, or a fixed point \fixed.  We set $u=$\open, $d=$\close, $r=$\lbounce, and $b$ equal to either \ubounce or \fixed (this notation corresponds to the code from the FindStat entry for the Corteel map).  Write a word $w$ whose $i$th letter $w_i$ is the letter corresponding to vertex $(i,i)$. Transform this to a Motzkin path where the letter $u$ corresponds to an up-step, $d$ a down-step, and $b$ and $r$ both correspond to a horizontal step.

  The nesting number $p_i$ is the number of arcs in the cycle diagram containing the vertex $(i,\sigma(i)),$ where containment is defined from above if $(i,\sigma(i))$ is on or above the diagonal (i.e. if $(i,i)$ is associated with $u$ or $b$), and below if $(i,\sigma(i))$ is below the diagonal (i.e. if $(i,i)$ is associated with $r$ or $d$).  More precisely,
  \[
    p_i= \begin{cases}
      \#\{j<i \ | \ \sigma(i)<\sigma(j)\}   & (i,i)\mbox{ is associated with } u \mbox{ or } b, \\
      \#\{i<j<n \ | \ \sigma(i)>\sigma(j)\} & (i,i)\mbox{ is associated with } r \mbox{ or } d.
    \end{cases}
  \]
  The colored Motzkin path corresponding to $\sigma$ is completely described by a word-weight pair $(\mathbf{w},\mathbf{p})$ where $\mathbf{w}$ is the associated word of length $n$ in $\{u,d,r,b\}$ and $\mathbf{p}$ is the vector of non-negative integer weights (or colors) given by the associated nesting numbers $p_i$.

  Let $h_i$ denote the height of step $i$ in the corresponding Motzkin path, which is defined to be the height of the step's lowest endpoint.  The complement of $(\mathbf{w},\mathbf{p})$ is the colored Motzkin path corresponding to $(\mathbf{w},\mathbf{\bar{p}})$ where $\bar{p_i}=h_i-p_i-1$ if $w_i=r$ and $h_i-p_i$ otherwise (see Figure \ref{fig: MotzkinPathEx}).  The \textbf{Corteel map} (Map 239) sends $\sigma$ to the permutation found by applying the inverse Foata-Zeilberger bijection, i.e.\ the unique permutation whose cycle diagram corresponds to $(\mathbf{w},\mathbf{\bar{p}})$.
\end{definition}

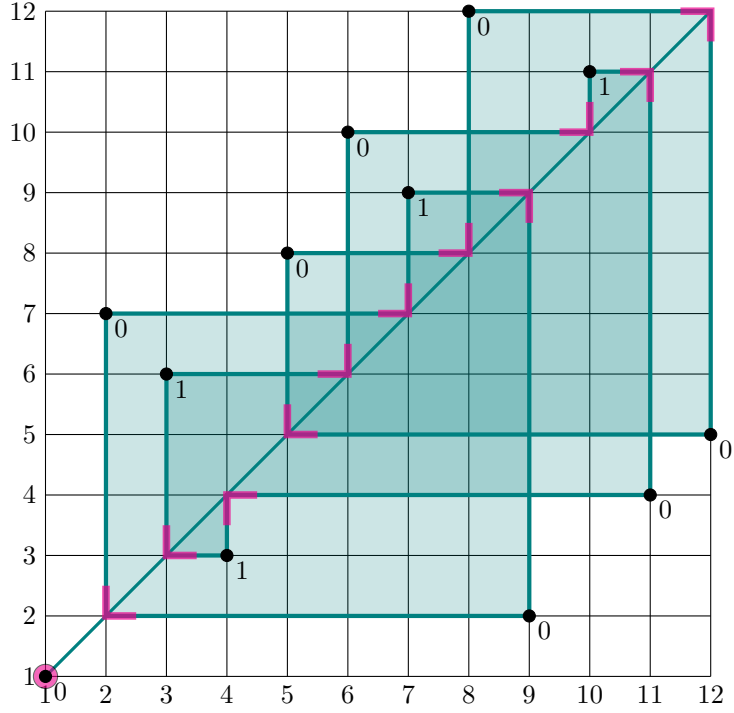
\begin{figure}[htbp]
  \centering
\begin{tikzpicture}[scale = .8]
    \def\myPerm{{1,7,6,3,8,10,9,12,2,11,4,5}}
    \def\myLabels{0,0,1,1,0,0,1,0,0,1,0,0}
    
    \foreach \i in {1,...,12} {
        \draw [thin,black] (\i,1) -- (\i,12)  node [below] at (\i,1) {$\i$};
    }
    \foreach \i in {1,...,12} {
        \draw [thin,black] (1,\i) -- (12,\i) node [left] at (1,\i) {$\i$};
    }
    
    \draw[very thick, teal] (1,1) -- (12,12);
    
    \draw[fill=teal,opacity=0.2] (2,2) -- (9,2) -- (9,9) -- (7,9) -- (7,7) -- (2,7) -- cycle;
    \draw[ultra thick, teal] (2,2) -- (9,2) -- (9,9) -- (7,9) -- (7,7) -- (2,7) -- cycle;

    \draw[fill=teal,opacity=0.2] (3,3) -- (4,3) -- (4,4) -- (11,4) -- (11,11) -- (10,11) -- (10,10) -- (6,10) -- (6,6) -- (3,6) -- cycle;
    \draw[ultra thick, teal] (3,3) -- (4,3) -- (4,4) -- (11,4) -- (11,11) -- (10,11) -- (10,10) -- (6,10) -- (6,6) -- (3,6) -- cycle;

    \draw[fill=teal,opacity=0.2] (5,5) -- (12,5) -- (12,12) -- (8,12) -- (8,8) -- (5,8) -- cycle;
    \draw[ultra thick, teal] (5,5) -- (12,5) -- (12,12) -- (8,12) -- (8,8) -- (5,8) -- cycle; 

    \def\myLoners{1}
    \def\myUppers{2,3,5}
    \def\myDowners{9,11,12}
    \def\myOvers{6,7,8,10}
    \def\myUnders{4}

    \foreach \i in \myLoners {
        \draw[fill = magenta,opacity=0.6] (\i,\i) circle (0.2cm);
    }
    \foreach \i in \myUppers {
        \draw[line width=1mm,magenta,opacity=0.7] (\i,\i.5) -- (\i,\i) -- (\i.5,\i);
    }
    \foreach \i in \myDowners {
        \draw[line width=1mm,magenta,opacity=0.7] (\i.5-1,\i) -- (\i,\i) -- (\i,\i.5-1);
    }
    \foreach \i in \myOvers {
        \draw[line width=1mm,magenta,opacity=0.7] (\i.5-1,\i) -- (\i,\i) -- (\i,\i.5);
    }    
    \foreach \i in \myUnders {
        \draw[line width=1mm,magenta,opacity=0.7] (\i,\i.5-1) -- (\i,\i) -- (\i.5,\i);
    }      

    \foreach [count=\i from 0] \j in \myLabels {
        \draw[fill=black] ({\i+1},{\myPerm[\i]}) circle (0.1cm);%
        \node [yshift=-.2cm,xshift=.2cm] at ({\i+1},{\myPerm[\i]}) {$\j$};
    }
\end{tikzpicture}
  \caption{Cycle diagram for the permutation $\sigma=1~7~6~3~8~10~9~12~2~11~4~5$.  Each node $(i,\sigma(i))$ is labeled with $p_i$, the number of cycles it is nested within in the cycle diagram.  From this diagram we see that the corresponding word-weight pair for this $\sigma$ is $(\mathbf{w},\mathbf{p})=((b,u,u,r,u,b,b,b,d,b,d,d),(0,0,1,1,0,0,1,0,0,1,0,0)).$}\label{fig: CycleDiagramEx}
\end{figure}

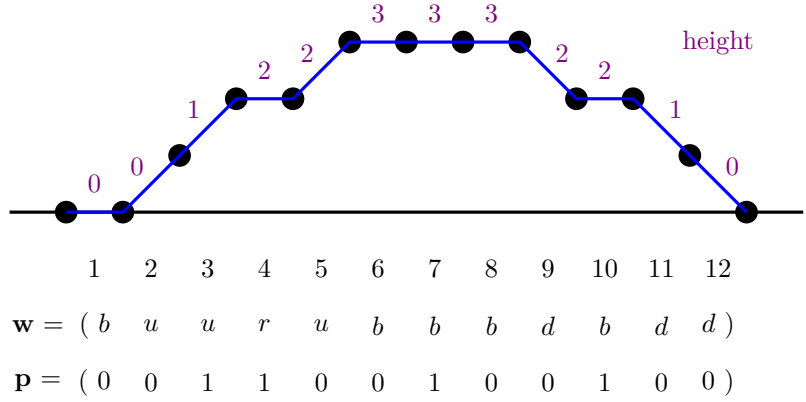
\begin{figure}[htbp]
  \centering
\begin{tikzpicture}[scale=.75]
 \draw (-0.5,-1) node foreach \x in {1,2,3,...,12} at (\x - 0.5,-1) {\x};
\draw (-0.5, -2) node {\textbf{w} =};
\draw ( .5, -2) node {( $b$};
\draw ( 1.5, -2) node {$u$};
\draw ( 2.5, -2) node {$u$};
\draw ( 3.5, -2) node {$r$};
\draw ( 4.5, -2) node {$u$};
\draw ( 5.5, -2) node {$b$};
\draw ( 6.5, -2) node {$b$};
\draw ( 7.5, -2) node {$b$};
\draw ( 8.5, -2) node {$d$};
\draw ( 9.5, -2) node {$b$};
\draw ( 10.5, -2) node {$d$};
\draw ( 11.5, -2) node {$d$ )};
\draw (-0.5, -3) node {\textbf{p} =};
\draw ( .5, -3) node {( 0};
\draw ( 1.5, -3) node {0};
\draw ( 2.5, -3) node {1};
\draw ( 3.5, -3) node {1};
\draw ( 4.5, -3) node {0};
\draw ( 5.5, -3) node {0};
\draw ( 6.5, -3) node {1};
\draw ( 7.5, -3) node {0};
\draw ( 8.5, -3) node {0};
\draw ( 9.5, -3) node {1};
\draw ( 10.5, -3) node {0};
\draw ( 11.5, -3) node {0 )};
\draw[violet] ( 0.5, 0.5) node {0};
\draw[violet] ( 1.25, .8) node {0};
\draw[violet] ( 2.25, 1.8) node {1};
\draw[violet] ( 3.5, 2.5) node {2};
\draw[violet] ( 4.25, 2.8) node {2};
\draw[violet] ( 5.5, 3.5) node {3};
\draw[violet] (6.5, 3.5) node {3};
\draw[violet] ( 7.5, 3.5) node {3};
\draw[violet] ( 8.75, 2.8) node {2};
\draw[violet] ( 9.5, 2.5) node {2};
\draw[violet] ( 10.75, 1.8) node {1};
\draw[violet] ( 11.75, .8) node {0};
\draw[violet] ( 11.5, 3) node {height};
 \fill[black] (0,0) circle (0.2cm) {};
\fill[black] (1,0) circle (0.2cm) {};
\fill[black] (2,1) circle (0.2cm) {};
 \fill[black] (3,2) circle (0.2cm) {};
  \fill[black] (4,2) circle (0.2cm) {};
\fill[black] (5,3) circle (0.2cm) {};
\fill[black] (6,3) circle (0.2cm) {};
\fill[black] (7,3) circle (0.2cm) {};
\fill[black] (8,3) circle (0.2cm) {};
 \fill[black] (9,2) circle (0.2cm) {};
 \fill[black] (10,2) circle (0.2cm) {};
  \fill[black] (11,1) circle (0.2cm) {};
 \fill[black] (12,0) circle (0.2cm) {};
 \draw[black, very thick] (-1, 0) -- (13, 0);
  \draw[blue, very thick] (0, 0) --  (1, 0) -- (2, 1)  -- (3, 2) -- (4, 2)  -- (5, 3) -- (6, 3) -- (7, 3)  -- (8, 3) -- (9, 2) -- (10, 2) --(11, 1)  -- (12, 0);
  \end{tikzpicture}
 \caption{Motzkin path for the cycle diagram given in Figure \ref{fig: CycleDiagramEx} The complement of the colored Motzkin path is $(\mathbf{w},\mathbf{\bar{p}})=((b, u, u, r, u, b, b, b, d, b, d, d),(0, 0, 0, 0, 2, 3, 2, 3, 2, 1, 1, 0)),$ which results in the permutation $1~10~12~2~7~6~9~8~5~11~4~3$.}
  \label{fig: MotzkinPathEx}
\end{figure}

The second involution which we show has $2^{n-1}$ fixed points is the invert Laguerre heap map (Definition~\ref{def: LaguerreHeap}). One way to define the invert Laguerre heap map is to associate a heap of pieces, as in \cite{Viennot_heaps}, to a given permutation by considering each decreasing run as one piece, beginning with the leftmost run. Two pieces commute if and only if the minimal element of one piece is larger than the maximal element of the other piece.  The invert Laguerre heap map returns the permutation corresponding to the heap obtained by reversing the reading direction of the heap. We use an equivalent definition in terms of non-crossing arcs, as described  in~\cite{reading2015noncrossing}.

\begin{definition}\label{def: LaguerreHeap}
  Given a permutation $\sigma=\sigma_1\sigma_2\cdots\sigma_n$,
  plot a point labeled $\sigma_i$ at $(i,\sigma_i)$, for each $1\leq i\leq n$. For each decreasing run with more than one element in it, draw a line between each pair of adjacent numbers within the decreasing run. Now move all the points to be aligned vertically, converting the lines connecting numbers into non-crossing arcs, and keeping all numbers on the correct side of the arcs. The \textbf{invert Laguerre heap map} (Map 241) proceeds by vertically reflecting this arc diagram and reconstructing the permutation by applying the backward map.
\end{definition}

\begin{example}
  Consider $n = 12$ and $\sigma=1~10~12~2~7~6~9~8~5~11~4~3$. 
  Applying the invert Laguerre heap map to $\sigma$ results in $1~11~4~3~9~8~5~7~6~12~2~10$ (see Figure \ref{fig: LaguerreHeap}). Since both the invert Laguerre heap and Corteel maps are involutions, comparing the output of the invert Laguerre heap map in this example to the starting permutation in Figure \ref{fig: CycleDiagramEx} shows these maps are not the same.

  Note that  applying the invert Laguerre heap map is not equivalent to recording the decreasing runs from right-to-left (preserving the order in each run), even though some examples may lead one to believe this.
\end{example}

\begin{figure}
  \begin{minipage}[c]{4.5cm}
\begin{tikzpicture}[scale=.35]
 \def\myPerm{{1,10,12,2,7,6,9,8,5,11,4,3}}
    \def\myLabels{1,10,12,2,7,6,9,8,5,11,4,3}
   
    \foreach \i in {1,...,12} {
        \draw [thin,gray] (\i,1) -- (\i,12);
    }
    \foreach \i in {1,...,12} {
        \draw [thin,gray] (1,\i) -- (12,\i);
    }
   
    \draw[very thick, Green] (3,12) -- (4,2);
    \draw[very thick, Purple] (5,7) -- (6,6);
    \draw[very thick, Cyan] (7,9) -- (8,8);
    \draw[very thick, MidnightBlue] (8,8) -- (9,5);
    \draw[very thick, OrangeRed] (10,11) -- (11,4);
    \draw[very thick, Orange] (11,4) -- (12,3);

    \foreach [count=\i from 0] \j in \myLabels {
        \draw[fill=black] ({\i+1},{\myPerm[\i]}) circle (0.075cm);%
        \node [yshift=-.2cm,xshift=-.2cm] at ({\i+1},{\myPerm[\i]}) {\tiny $\j$};
    }
    \draw[fill=Yellow](1,1) circle(0.15cm);
    \draw[fill=teal](2,10) circle(0.15cm);
\end{tikzpicture}
  \end{minipage}
\begin{minipage}[l]{.4cm}$\rightarrow$\end{minipage}
  \begin{minipage}[c]{2cm}
\begin{tikzpicture}[scale=.35]

   \draw [thin,lightgray] (1,1) -- (1,12);

   \foreach \i in {1,2,3,4,5,8,9,11,12} {
        \draw[fill=black] (1,{\i}) circle (0.1cm);%
        \node [xshift=.2cm] at ({1},{\i}) {\tiny $\i$};
    }
    \foreach \i in {6,7,10} {
        \draw[fill=black] (1,{\i}) circle (0.1cm);%
        \node [xshift=-.25cm] at ({1},{\i}) {\tiny $\i$};
    }
    \draw[thick, Green] (1,2) to[curve through = { (0.5,3) (1,9.5) (1.5,10) (1,10.5) (0.5,11) }] (1,12);
    \draw[thick, MidnightBlue] (1,5) to[curve through = { (1.5,6.25) }] (1,8);
    \draw[thick, OrangeRed] (1,4) to[curve through = { (2.5,7.75) }] (1,11);
    \draw[very thick,Cyan] (1,8)--(1,9);
    \draw[very thick,Purple] (1,6)--(1,7);
    \draw[very thick, Orange] (1,3)--(1,4);

    \draw[fill=Yellow](1,1) circle(0.15cm);
    \draw[fill=teal](1,10) circle(0.15cm);
\end{tikzpicture}
\end{minipage}
\begin{minipage}[l]{.4cm}$\rightarrow$\end{minipage}
  \begin{minipage}[c]{2cm}
\begin{tikzpicture}[scale=.35]

   \draw [thin,lightgray] (2,1) -- (2,12);

   \foreach \i in {1,2,3,4,5,8,9,11,12} {
        \draw[fill=black] (2,{\i}) circle (0.1cm);%
        \node [xshift=-.25cm] at ({2},{\i}) {\tiny $\i$};
    }
    \foreach \i in {6,7,10} {
        \draw[fill=black] (2,{\i}) circle (0.1cm);%
        \node [xshift=.2cm] at ({2},{\i}) {\tiny $\i$};
    }
    \draw[thick, Green] (2,2) to[curve through = { (2.5,3) (2,9.5) (1.5,10) (2,10.5) (2.5,11) }] (2,12);
    \draw[thick, MidnightBlue] (2,5) to[curve through = { (1.5,6.25) }] (2,8);
    \draw[thick, OrangeRed] (2,4) to[curve through = { (0.5,7.75) }] (2,11);
    \draw[very thick,Cyan] (2,8)--(2,9);
    \draw[very thick,Purple] (2,6)--(2,7);
    \draw[very thick, Orange] (2,3)--(2,4);
    
    \draw[fill=Yellow](2,1) circle(0.15cm);
    \draw[fill=teal](2,10) circle(0.15cm);

\end{tikzpicture}
\end{minipage}
\begin{minipage}[l]{.4cm}$\rightarrow$\end{minipage}
\begin{minipage}[c]{4.5cm}
\begin{tikzpicture}[scale=.35]
 \def\myPerm{{1,11,4,3,9,8,5,7,6,12,2,10}}
    \def\myLabels{1,11,4,3,9,8,5,7,6,12,2,10}
   
    \foreach \i in {1,...,12} {
        \draw [thin,gray] (\i,1) -- (\i,12);
    }
    \foreach \i in {1,...,12} {
        \draw [thin,gray] (1,\i) -- (12,\i);
    }
   
    \draw[very thick, Green] (10,12) -- (11,2);
    \draw[very thick, Purple] (8,7) -- (9,6);
    \draw[very thick, Cyan] (5,9) -- (6,8);
    \draw[very thick, MidnightBlue] (6,8) -- (7,5);
    \draw[very thick, OrangeRed] (2,11) -- (3,4);
    \draw[very thick, Orange] (3,4) -- (4,3);

    \foreach [count=\i from 0] \j in \myLabels {
        \draw[fill=black] ({\i+1},{\myPerm[\i]}) circle (0.075cm);%
        \node [yshift=-.2cm,xshift=-.2cm] at ({\i+1},{\myPerm[\i]}) {\tiny $\j$};
    }
    \draw[fill=Yellow](1,1) circle(0.15cm);
    \draw[fill=teal](12,10) circle(0.15cm);
\end{tikzpicture}
\end{minipage}
  \caption{The invert Laguerre heap map applied to the permutation $1~10~12~2~7~6~9~8~5~11~4~3$. The resulting permutation is $1~11~4~3~9~8~5~7~6~12~2~10.$ }\label{fig: LaguerreHeap} 
\end{figure}
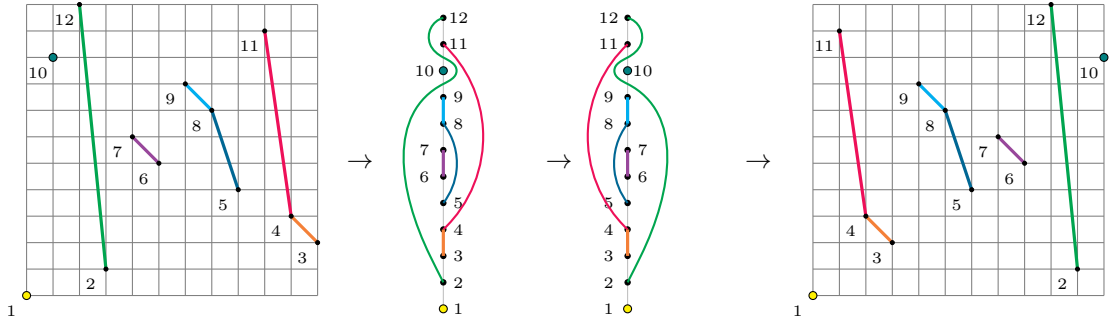

Before proving our  CSP theorems, we establish the number of fixed points of the Corteel and invert Laguerre heap maps in the following lemmas. We note that though the orbit structure of the  maps match, the orbit elements differ.  In particular, while there is a large overlap in the fixed points of each map, they are not all the same, and we thus have to prove the number of fixed points result separately.

\begin{lem}\label{lem: Corteel fixed pts}  The number of fixed points under the action of the Corteel map on $S_n$ is $2^{n-1}.$
\end{lem}

\begin{proof}
  By definition, a fixed point of the Corteel map corresponds to a colored Motzkin path which is equal to its complement relative to the height of the path.
  Recall that the complement of $(\mathbf{w},\mathbf{p})$ is the colored Motzkin path corresponding to $(\mathbf{w},\mathbf{\bar{p}})$ where $\bar{p_i}=h_i-p_i-1$ if $w_i=r$ and $h_i-p_i$ otherwise.
  Thus, if the colored Motzkin path is fixed under the complement (i.e. $\mathbf{\bar{p}}=\mathbf{p}$) it must be that $h_i-1=2p_i$ when $w_i=r$ and $h_i=2p_i$ otherwise.  Suppose $h_i=1.$  Since fractional values for $p_i$ are not allowed, it follows $w_i=r$ and the left endpoint of the $(i+1)$-th step has height $1.$  By definition of the height function, $h_{i+1}$ must be less than or equal to 1.  As the height at each step can only change by $0$ or $\pm 1$, $h_1=0,$ and we just showed that if $h_i=1$ then $h_{i+1}\leq 1$, it follows that at each step $h_i=0$ (and $w_i\in \{u,b,d\}$) or $h_i=1$ (and $w_i=r$), and in either case $p_i=\bar{p_i}=0.$

  The only valid cycle diagrams satisfying these height conditions are those corresponding to words $\mathbf{w}$ which start with the letter $u$ or $b$, end with the letter $b$ or $d$, have an equal number of $u$'s and $d$'s, and for all $i>1$ satisfy the conditions:
  \begin{itemize}
    \item $w_i\in \{u,b\}$ if and only if $w_{i-1}\in\{b,d\},$ and
    \item $w_i\in\{r,d\}$ if and only if $w_{i-1}\in\{r,u\}.$
  \end{itemize}
  These conditions completely determine the last letter of the word $\mathbf{w}$, but for each other letter we can choose between two possible values ($\{u,b\}$ or $\{r,d\}$) depending on what letter precedes it.  Thus there are $2^{n-1}$ possible words, and hence $2^{n-1}$ colored Motzkin paths fixed under the complement, and $2^{n-1}$ permutations fixed by the action of the Corteel map.
\end{proof}

From the preceding proof it follows that the permutations fixed by the Corteel map correspond to cycle diagrams made up of disjoint squares ($ud$ pairs), stairs ($urr\cdots rd$), and fixed points ($b$).  See Figure \ref{fig: SquaresStairs}.

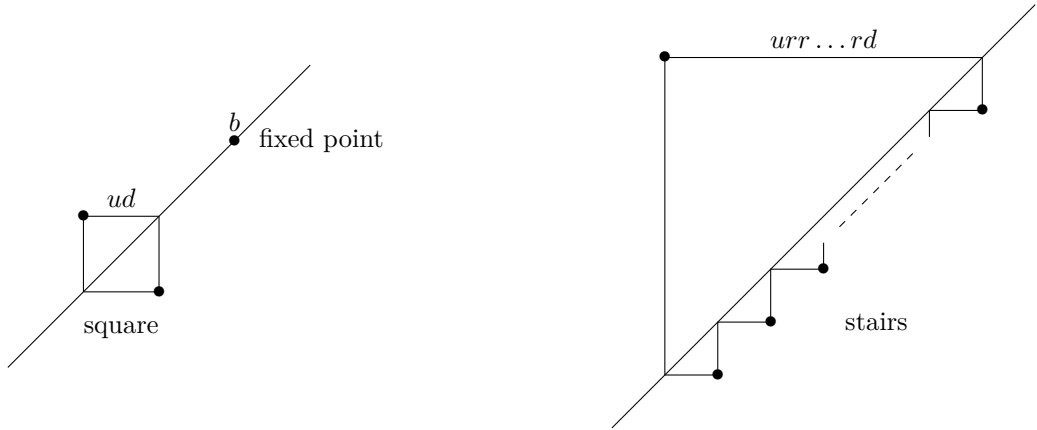
\begin{figure}[htbp]
  \centering
  \begin{minipage}{0.4\linewidth}
            \centering
            \begin{tikzpicture}
                \draw (0,0) to (4,4); 
                \draw (1,1) to (1,2);
                \draw (2,2) to (1,2);
                \draw (1.5,2) node[above]{$ud$};
                \draw (1,1) to (2,1);
                \draw (2,2) to (2,1);
                \node at (1,2) {$\bullet$};
                \node at (2,1) {$\bullet$};
                \draw (3,3) node {$\bullet$} ;
                \draw (3,3) node [above] {$b$} ;
                \draw (3.2,3) node [right] {fixed point};
                \draw (1.5,0.5) node {square};
            \end{tikzpicture}
        \end{minipage}
        \qquad
        \begin{minipage}{0.5\linewidth}
            \centering
            \begin{tikzpicture}[scale=0.7]
                \draw (0,0) to (8,8); 
                \draw (1,1)--(2,1)--(2,2)--(3,2) -- (3,3) -- (4,3) --(4,3.5);
                \draw[dashed] (4.3,3.8) -- (5.7, 5.2);
                \draw (6, 5.5) --(6,6) -- (7,6) -- (7,7) -- (1,7) -- (1,1);
                \node at (5,2) {stairs};
                \node at (1,7) {$\bullet$};
                \node at (2,1) {$\bullet$};
                \node at (3,2) {$\bullet$};
                \node at (4,3) {$\bullet$};
                \node at (7,6) {$\bullet$};
                \draw (4,7) node[above]{$urr\ldots rd$};
            \end{tikzpicture}
        \end{minipage}
  \caption{Permutations fixed by the Corteel map correspond to cycle diagrams made up of disjoint squares, stairs, or fixed points along the diagonal.}
  \label{fig: SquaresStairs}
\end{figure}

\begin{lem}\label{lem: Laguerre fixed pts}  The number of fixed points under the action of the invert Laguerre heap map on $S_n$ is $2^{n-1}.$
\end{lem}

\begin{proof}
  The invert Laguerre heap map corresponds to vertically flipping the arc diagrams of \cite{reading2015noncrossing}. The arc diagrams that are fixed under this flip are the ones whose only arcs are between adjacent numbers. Each such arc can either be there or not, so the total number is $2^{n-1}$.
\end{proof}

In each of the following subsections we prove a given statistic exhibits the CSP.  We are able to verify futher instances of the CSP along the way by citing or proving equidistribution results.  In total, we establish the CSP for $13$ statistics, and conjecture it for two more, with respect to involutions having $2^{n-1}$ fixed points. 

\subsection{The number of crossings}\label{subsec: CorteelCSP}  In this subsection we prove the CSP for the number of crossings directly, and then extend this result to the number of nestings and the number of occurrences of $13-2$ and $31-2$ patterns.  
In addition to the statistics defined here, recall consecutive patterns defined in Definition~\ref{def:patterns}.
\begin{definition}\label{def:crossings_nestings}
  A \textbf{crossing} of a permutation $\sigma\in S_n$ is a pair $(i,j)$
  such that either $i<j\leq \sigma(i)< \sigma(j)$ or $\sigma(i)<\sigma(j)<i<j$.
  Using the cycle diagram described in Definition~\ref{def: Corteel}, the number of crossings of the permutation is the number of crossings of arcs in the diagram plus the number of times a line above the $y=x$ diagonal bounces off it.

  A \textbf{nesting} of $\sigma$ is a pair $(i,j)$ such that $j < i \leq  \sigma(i) < \sigma(j)$ or $\sigma(j) < \sigma(i) < i < j$. In the cycle diagram, these are a pair of nested arcs above or below the $y=x$ diagonal, including the nesting of fixed points with respect to arcs above the diagonal. The number of nestings is the sum of the nesting numbers $p_i$ from Definition~\ref{def: Corteel}.
\end{definition}

\begin{thm}\label{thm: Corteel-crossings CSP}
  The number of crossings of a permutation (Statistic $39$) exhibits the cyclic sieving phenomenon with respect to involutions having $2^{n-1}$ fixed points.
\end{thm}

\begin{proof}
  By \cite[Theorem 2]{corteel2007crossings}, the generating function for the number of crossings statistic is given by:
  \[\sum_{k=1}^n\hat{E}_{k,n}(q)\] where
  \[\hat{E}_{k,n}(q)=q^{k-k^2}\sum_{i=0}^{k-1}(-1)^i[k-i]_q^n q^{k(i-1)}\left(\binom{n}{i}q^{k-i}+\binom{n}{i-1}\right).
  \]
  By \cite[Proposition 5.7]{LW2005}, $\hat{E}_{k,n}(-1)=\binom{n-1}{k-1}$, thus $\sum_{k=1}^n\hat{E}_{k,n}(-1)=\sum_{k=1}^n\binom{n-1}{k-1}=2^{n-1}$.
\end{proof}

\begin{cor}\label{cor: Corteel-crossings CSP}
  The following statistics exhibit the CSP under involutions with $2^{n-1}$ fixed points.
  \begin{itemize}
    \item Statistic $223$: The number of nestings,
    \item Statistic $356$: The number of occurrences of the pattern $13-2$,
    \item Statistic $358$: The number of occurrences of the pattern $31-2$.
  \end{itemize}
\end{cor}

\begin{proof}
  By \cite[Proposition 4]{corteel2007crossings}, the number of nestings of a permutation has the same generating function as the number of crossings, so the CSP holds for this statistic as well.

  By \cite[Corollary 30]{SteingrimssonWilliams2007}, the number of permutations of $[n]$ with $k-1$ descents and $m$ occurrences of the pattern $2-31$ is given by the coefficient of $q^m$ in $\hat{E}_{k,n}(q)$ (from the proof of Theorem~\ref{thm: Corteel-crossings CSP}). So the generating function for occurrences of the pattern $2-31$ is $\sum_{k=1}^{n}\hat{E}_{k,n}(q)$, which is the generating function for the number of crossings given in the proof of Theorem~\ref{thm: Corteel-crossings CSP}. Applying the reverse map sends occurrences of the pattern $2-31$ to occurrences of the pattern $13-2$, so we have that the number of occurrences of the pattern $13-2$ is equidistributed with the number of crossings. Thus the desired CSP holds.

  Finally, $13-2$ patterns and $31-2$ patterns are related by the complement map and so are equidistributed. Thus we have the CSP for all three statistics follows from the CSP of the number of crossings.
\end{proof}

Note that FindStat indicates that the number of occurrences of the pattern $31-2$ and the number of crossings may be related by the Clark-Steingrimsson-Zeng map (Map 238) \cite{ClarkeSteingrimssonZeng} followed by the inverse map. It may be interesting to prove the equidistribution of the number of $31-2$ patterns and crossings directly via these maps.

\subsection{The cycle descent number}\label{subsec:Corteel descent}
We first establish the CSP for the cycle descent number with any involution having $2^{n-1}$ fixed points. We then prove $\arrowpattern$ arrow patterns are equidistributed with the cycle descent number, and thus also exhibit the CSP. 

\begin{definition}\label{def: cycle descent} A \textbf{cycle descent} of a permutation $\pi$ is a descent within a cycle of $\pi$ when $\pi$ is written in cycle notation with each cycle starting with its smallest element.
  The \textbf{cycle descent number}, denoted $\cdes(\pi)$, counts the number of cycle descents in $\pi.$
\end{definition}
Note that a cycle must have at least length $3$ to have a cycle descent; this is because the first number in a cycle is the smallest, so it does not form a cycle descent with the next number in the cycle.

\begin{thm}\label{thm:cycle_des_CSP}
  The cycle descent number (Statistic $317$) exhibits the CSP with respect to involutions having $2^{n-1}$ fixed points.
\end{thm}
\begin{proof}
  Recall $\cdes(\pi)$ denotes the cycle descent number statistic on the permutation $\pi$.
  By \cite[Theorem 1]{CycleDescent} (setting $x=1$ and summing over all $i$), we have
  \[\sum_{\pi\in\mathfrak{S}_n} (-1)^{\cdes(\pi)}t^{\pi^{-1}(1)}=2^{n-2}(t+t^n).
  \]
  Setting $t=1$ gives that the generating function of the cycle descent number evaluated at $-1$ equals $2^{n-1}$. Thus, the cycle descent number exhibits the desired CSP.
\end{proof}

\begin{definition}
  Given a permutation $\pi$, a \textbf{12 arrow pattern}, denoted  $\arrowpattern$, is an instance of a $12-$ pattern, that is, an ascent $\pi_i,\pi_{i+1}$, with the special property that when you apply the first fundamental transform (Definition \ref{def:fund_transform}) to create a new permutation $\sigma=\mathcal{F}(\pi)$, you have $\sigma(\pi_i)=\pi_{i+1}$.
\end{definition}

See~\cite[Definition 17]{TennerArrow} for the definition of arrow patterns. Since this is a very special case, we do not need to define these in full generality.

\begin{example}
  Let $\pi=72358164$.  The following $(\pi_1,\pi_{i+1})$ pairs form $12-$ patterns in $\pi: (2,3), (3,5), (5,8),$ and $(1,6)$.  To determine which of these are $\arrowpattern$ arrow patterns, calculate the
  first fundamental transform of $\pi$ by inserting $)($ before each left-to-right maximum to divide $\pi$ into cycles.  Writing the resulting permutation,  $(7235)(8164),$ in one-line notation gives $\sigma=\mathcal{F}(\pi)=63587421.$  Since $\sigma(5)=7\not=8$, $(5,8)$ does not form a $\arrowpattern$ arrow pattern, but all other $12-$ pairs in this example do.
\end{example}

Note that a 12 pattern in a permutation $\pi$ given by $(\pi_i, \pi_{i + 1})$ is a 12 arrow pattern if $\pi_{i + 1}$ is not a left-to-right maximum, and no left-to-right maximum can contribute to a 12 arrow pattern.

\begin{lem}\label{lem:317_373_1744} The number of occurrences of the arrow pattern $\arrowpattern$ (Statistic $1744$) is equidistributed with the cycle descent number.
 \end{lem}

\begin{proof}
  Let $\Stat1744(\pi)$ denote the number of occurrences of the arrow pattern $\arrowpattern$ in $\pi$.

  We prove the first fundamental transform (see Definition~\ref{def:fund_transform}), followed by the inverse map, is a bijection that maps  Statistic 1744 to the number of cycle descents.

  Let $\pi\in S_n$ and $\sigma$ denote the permutation obtained by applying the the first fundamental transform to $\pi$. Let $a,b$ be an ascent of $\pi$ such that  $\sigma(a)=b$, i.e. an instance of a $\arrowpattern$ arrow pattern. It must be that $b$ is not a left-to-right maximum of $\pi$, since if it were, by definition of $\mathcal{F}$, $a$ and $b$ would be in different cycles of $\sigma$ and it would not be true that $\sigma(a)=b$.  Thus, $a$ and $b$ are consecutive inside a cycle of $\sigma$ with the largest entry $c$ of that cycle written first. So the cycle looks like $(c,\ldots,a,b,\ldots,d)$ with $a<b<c$ and $d$ possibly equal to $b$.

  Having established that all instances of a $\arrowpattern$ arrow pattern appear as consecutive elements in some cycle of $\sigma$, we argue that the arrow pattern $a,b$ in $\pi$ corresponds to a unique cycle descent in $\sigma^{-1}$, namely the cycle descent $c,d$ when $a$ is the smallest element in the cycle, and $b,a$ otherwise.
  To see this, take the inverse of $\sigma$. This results in reversing everything after $c$ in the cycle, resulting in the cycle $(c,d,\ldots,b,a,\ldots)$. To calculate the cycle descent number, we need to rewrite each cycle with the smallest number first.  Let $s$ be this smallest element and suppose $a\not=s$  Then rewriting the cycle of $\sigma^{-1}$ so that $s$ appears first we have $(s,\ldots,b,a,
    \ldots)$ and $b,a$ is a cycle descent of $\sigma^{-1}$.

  If $s=a$, then rewriting the cycle of $\sigma^{-1}$ so that it starts with the smallest element we have $(a,\ldots,c,d,\ldots, b)$.  In this case, $b,a$ is not a descent of $\sigma^{-1}$, but $c,d$ is.  And since $c$ was the first element, and maximal, in the cycle of $\sigma$, it was not part of a $\arrowpattern$ arrow pattern in $\pi.$
  A similar argument in reverse shows that every cycle descent in $\sigma^{-1}$ that does not involve the maximal element corresponds directly to a $\arrowpattern$ arrow patterns in $\pi$, and a cycle descent in $\sigma^{-1}$ involving the maximal element corresponds to the $\arrowpattern$ arrow patterns in $\pi$ involving the minimal cycle element.
  Thus every $\arrowpattern$ arrow pattern contributing to Statistic 1744 on $\pi$ corresponds to a cycle descent of $\sigma^{-1}$, and vice versa, establishing the result $\Stat1744(\pi)=\cdes(\sigma^{-1})$.
 \end{proof}

\begin{cor}\label{cor:Corteel_arrow_CSP}
The number of occurrences of the arrow pattern $\arrowpattern$ (Statistic $1744$) exhibits the CSP with respect to involutions having $2^{n-1}$ fixed points.
\end{cor}


\subsection{The  number of midpoints of decreasing subsequences of length 3 }\label{subsec:Corteel midpoints}

We establish the CSP for the number of midpoints of decreasing and increasing subsequences of length $3$, as well as the number of distinct positions of the pattern letter $3$ in $132$ patterns, and the number of distinct positions of the letter $2$ in $213$ patterns.   

\begin{thm}
  \label{conj:371}
  The number of midpoints of decreasing subsequences of length $3$ (Statistic $371$) exhibits the CSP with respect to involutions having $2^{n-1}$ fixed points.  
\end{thm}

\begin{proof}
  Define an involution $\psi$ on $S_n$ with $2^{n-1}$ fixed points, where each fixed point contributes 0 to Statistic 371, and if $\sigma$ is not a fixed point then the number of midpoints of decreasing subsequences in $\sigma$ and $\psi(\sigma)$ differ by $\pm 1,$ as follows.  Let $3**$ denote a pattern of the form $321$ or $312$. Let $j$ be the maximal midpoint index over all such patterns (if they exist), i.e.\ $j$  is maximal such that there exists $i, k$ with $i < j < k$ and $\sigma_i > \sigma_j , \sigma_k$. Thus $j$ is the maximum index corresponding to $2$ in $321$ patterns or $1$ in $312$ patterns. If $i<j<k$ and $i<j<k'$ are both $3**$ patterns with $k<k'$ then $i<k<k'$ forms another $3**$ pattern with midpoint $k$ greater than $j$.  Thus, $3**$ patterns corresponding to \emph{maximum} midpoint index $j$ have not only unique midpoint index $j$, but also unique endpoint index $k$.

  We define $\psi$ as follows:
  \begin{itemize}
    \item If $\sigma$ contains a $3**$ pattern, set $\psi(\sigma)$ to be the permutation created from $\sigma$ by swapping $\sigma_j$ and $\sigma_k$.  In other words, $\psi(\sigma)$ converts the $3**$ pattern with maximum midpoint index $j$ from $321$ to $312$, or vice versa.
    \item If $\sigma$ avoids the patterns  $321$ and $312$, then $\psi(\sigma) = \sigma$.
  \end{itemize}

  Clearly, $\sigma$ is a fixed point of $\psi$ if and only if it avoids the patterns $321$ and $312$.  By \cite{SimionSchmidt}, there are exactly $2^{n-1}$ such permutations of $[n]$.   Otherwise, by the definition of $\psi$, the maximal midpoint index $j$ of a $3**$ pattern in $\sigma$ is the same as that in $\psi(\sigma),$  only $\psi(\sigma)_j$ corresponds to $1$ in the pattern if $\sigma_j$ corresponds to $2$ and $\psi(\sigma)_j$ corresponds to $2$ in the pattern if $\sigma_j$ corresponds to $1.$  Applying $\psi$ to $\psi(\sigma)$ swaps back the $312$ and $321$ patterns interchanged by $\psi$.  Thus $\psi(\psi(\sigma))=\sigma$, establishing $\psi$ is an involution with $2^{n-1}$ fixed points.

  All that remains to show is that $\sigma$ and $\psi(\sigma)$ have a number of midpoints of decreasing subsequences of length $3$ that differ by $\pm 1$. All decreasing subsequences of length three that do not involve $\sigma_j$ and $\sigma_k$ are unchanged by $\psi.$  Thus we only need to examine those patterns that include $\sigma_j$, $\sigma_k$, or both.

  We start by noting that for all indices $s>j$ but not equal to $k$, $\sigma_s>\sigma_i$ since otherwise $i,k,s$ would form a $3**$ pattern contradicting our assumption that $j$ was the maximal midpoint of such patterns.  Thus, there are no decreasing subsequences of length 3 in $\sigma$ or $\psi(\sigma)$ starting at index $j$ or $k$, or including index $k$ as a midpoint.
  Next consider length 3 decreasing subsequences in $\sigma$ with indices $i'<j'<k',$ ending at $j$ or $k$, i.e.\ with $k'=j$ or $k.$ For all such subsequences, $j'\leq j$ since otherwise $\sigma_{i'},\sigma_{j'},\sigma_{k'}$ would form a $3**$ pattern with midpoint index $j'$ greater than the maximum $j$.  If $j'<j$ and $k'=k$ or $j$ then swapping $\sigma_j$ and $\sigma_k$ exchanges the decreasing subsequence $\sigma_{i'},\sigma_{j'}, \sigma_{k}$ with $\psi(\sigma)_{i'},\psi(\sigma)_{j'}, \psi(\sigma)_j$, or $\sigma_{i'},\sigma_{j'}, \sigma_{j}$ with $\psi(\sigma)_{i'},\psi(\sigma)_{j'}, \psi(\sigma)_k.$  In either case, $j'$ is still a midpoint of a decreasing subsequence of length three and no length three subsequence midpoints have been added or destroyed.

  The only length 3 decreasing subsequences left to consider are those with midpoint index $j$.  Recall that the endpoint index $k$ of a $3**$ pattern with maximum midpoint index $j$ is unique.  Thus switching $\sigma_j$ and $\sigma_k$ either breaks all decreasing subsequences of length $3$ with midpoint index $j$, or creates a decreasing subsequence of length 3 in $\psi(\sigma)$ with midpoint $j$ when $j$ was not a midpoint of such a decreasing subsequence before.  More precisely, if $i<j<k$ are the indices of a $3**$ pattern in $\sigma$, with $\sigma_j > \sigma_k$, then $i,j,k$ forms a $321$ pattern in $\sigma$ that breaks in $\psi(\sigma)$; and if $\sigma_j < \sigma_k$, then $i,j,k$ forms a $312$ pattern in $\sigma$ that becomes a $321$ pattern in $\psi(\sigma)$. Thus, in swapping $\sigma_j$ and $\sigma_k$, the number of midpoints of decreasing subsequences of length three is changed by $\pm 1.$

  This shows that Statistic $371$ exhibits the CSP with respect to  any involution with $2^{n-1}$ fixed points.
\end{proof}

\begin{example}\label{ex: 3**}
  The following is an example of $\psi$ from the proof of Theorem \ref{conj:371}. Let $n = 6$ and $\sigma = 251346$. Then $\sigma$ has no midpoints of decreasing subsequences of length $3$ and $j = 4, k = 5$ since $\sigma_4 = 3, \sigma_5 = 4 < \sigma_2 = 5$. So, $\psi(\sigma) = 251436$, which has one midpoint of decreasing subsequences of length $3$ at spot $4$.

  Next, let $\sigma = 352461$. Then $\sigma$ has two midpoints of decreasing subsequences of length $3$ and $j = 4, k = 6$ since $\sigma_4 = 4, \sigma_6 = 1 < \sigma_2 = 5$. So, $\psi(\sigma) = 352164$, which has one midpoint of decreasing subsequences of length $3$ as it lost the one at spot $4$.
\end{example}

Again, recall Definition~\ref{def:patterns} for permutation patterns.
\begin{lem}\label{lem:equid_371}
  The following statistics are equidistributed:
  \begin{itemize}
    \item  Statistic $371$: The number of midpoints of decreasing subsequences of length $3$,
    \item Statistic $372$: The number of midpoints of increasing subsequences of length $3$,
    \item  Statistic $1683$: The number of distinct positions of the pattern letter $3$ in occurrences of $132$,
    \item Statistic $1687$: The number of distinct positions of the pattern letter $2$ in occurrences of $213$.
  \end{itemize}
\end{lem}
\begin{proof}
  Take the pattern $132$, in which we follow the entry that is initially the pattern letter $3$. Its positions correspond to Statistic 1683. Write $1\,\circlesign{3}\,2$ for this pattern, and the circled entry is the one that we follow. Through complement, inverse and complement again, this position is sent to the $2$ of the pattern $213$:
  \[ 1\,\circlesign{3}\,2 \xmapsto{\mathcal{C}} 3\,\circlesign{1}\,2 \xmapsto{\mathcal{I}} \circlesign{2}\,3\,1 \xmapsto{\mathcal{C}} \circlesign{2}\,1\,3.\]
  Hence, the number of distinct positions of the pattern letter $3$ in $132$ is equidistributed with the number of distinct positions of the pattern letter $2$ in $213$ (Statistic 1687).
  The number of midpoints of increasing (resp.\ decreasing) subsequences of length $3$ corresponds to the patterns $1\,\circlesign{2}\,{3}$ and $3\,\circlesign{2}\,{1}$ respectively in the language above. Theorem $2$ of \cite{ThamrongpairojRemmel} states the equidistribution of $1\,\circlesign{2}\,{3}$ and $1\,\circlesign{3}\,{2}$. Because $\C(1\,\circlesign{2}\,{3}) = 3\,\circlesign{2}\,{1}$, the number of midpoints of decreasing sequences is also equidistributed with the statistics above. Thus Statistics $371$ and $372$ are also equidistributed.
\end{proof}

\begin{cor}\label{cor:Corteel_mid_decreasing}
  The following statistics exhibit the CSP under involutions with $2^{n-1}$ fixed points.
  \begin{itemize}
    \item Statistic $372$: The number of midpoints of increasing subsequences of length $3$,
    \item  Statistic $1683$: The number of distinct positions of the pattern letter $3$ in occurrences of $132$ in a permutation,
    \item Statistic $1687$: The number of distinct positions of the pattern letter $2$ in occurrences of $213$,
  \end{itemize}
\end{cor}

\begin{proof}
  Lemma \ref{lem:equid_371} shows that Statistic $371$ is equidistributed with Statistics $372$, $1683$, and $1687$, and thus by Theorem~\ref{conj:371} these statistics also exhibit the CSP under involutions with $2^{n-1}$ fixed points.
\end{proof}
 
\begin{cor}\label{cor:Corteel_1004}
  The number of indices that are either left-to-right maxima or right-to-left minima (Statistic 1004) exhibits the CSP under involutions with $2^{n-1}$ fixed points for even values of $n$, but not for odd values of $n \geq 3$.
\end{cor}
\begin{proof}
  Note that the generating function for this statistic is the complement with respect to $n$ of Statistic $371,$ the number of midpoints of decreasing subsequences of length $3$, meaning that Statistic $371$ consists of the indices that are neither left-to-right maxima nor right-to-left minima. Then, $\mathrm{Stat}1004 = n - \mathrm{Stat}371$, and the statistic generating function of Statistic 1004 is
  \[ \sum_{\sigma\in S_n} q^{\mathrm{Stat}1004(\sigma) } = \sum_{\sigma\in S_n} q^{n-\mathrm{Stat}371(\sigma)},\]
  whose evaluation at $-1$ is
  \[ \sum_{\sigma\in S_n} (-1)^{\mathrm{Stat}1004(\sigma) } = \sum_{\sigma\in S_n} (-1)^n (-1)^{\mathrm{Stat}371(\sigma)}= (-1)^n \sum_{\sigma\in S_n} (-1)^{\mathrm{Stat}371(\sigma)}.\]
  Therefore, when $n$ is even, the evaluations of the statistic generating functions for Statistics $371$ and $1004$ at $1$ and $-1$ agree. Following Theorem~\ref{conj:371}, the number of indices that are either left-to-right maxima or right-to-left minima exhibits the CSP for involutions with $2^{n-1}$ fixed points when $n$ is even.
  
  When $n$ is odd, the evaluation at $-1$ is negative, so the CSP does not occur.
\end{proof}


\subsection{The  number of occurrences of the pattern 32-1}\label{subsec:Corteel 321} 


To prove the CSP for the number of occurrences of the pattern $32 - 1$ and the pattern $12-3$, we define and study yet another involution with $2^{n-1}$ fixed points. 
Recall Definition~\ref{def:patterns}.
\begin{thm}\label{thm:357}
The number of occurrences of the pattern $32-1$ (Statistic $360$) exhibits the CSP with respect to involutions having $2^{n-1}$ fixed points.  
\end{thm}

\begin{proof}
  We start by defining an involution $\psi$ on the set of permutations of $[n].$ For $n \leq 2$, define $\psi$ as the identity map.  For $n = 3$, define $\psi$ as pairing $312$ to $321$ and mapping everything else to itself.
For $n > 3$, define $\psi$ as follows:
  \begin{itemize}
    \item If $\sigma_1 \neq 1, 2$, $\psi(\sigma)$ is formed from $\sigma$ by swapping the values $1$ and $2$.
    \item If $\sigma_1 =1$ or $\sigma_1=2$, let $\sigma'$ be the permutation of $[n-1]$ defined by $\sigma'_i = \sigma_{i + 1} - 1$ if $\sigma_{i+1} \neq 1$ and $\sigma'_i = 1$ otherwise. Define $\psi(\sigma)$ by setting $\psi(\sigma)_1 = \sigma_1$. If $\sigma_1 = 1$, for $i > 1$, $\psi(\sigma)_i = \psi(\sigma')_{i-1} + 1$. If $\sigma_1 = 2,$ for $i > 1$, $\psi(\sigma)_i = \psi(\sigma')_{i-1} + 1$ if $\psi(\sigma')_{i-1} \neq 1$ and $\psi(\sigma)_i = 1$ otherwise.
  \end{itemize}
  Note that while $\psi$ is an involution, it is not the same map that was given in the proof of Theorem \ref{conj:371}.  The number of fixed points of $\psi$ for $n \leq 3$ is $2^{n-1}$. For $n > 3$, the number of fixed points is $2$ times the number of fixed points for $n - 1$, so again we have $2^{n-1}$. In addition, one can verify that any fixed point for $n \leq 3$ contributes 0 to the statistic. For $n > 3$, any fixed point has $\sigma_1 = 1$ or $2$ so $\sigma_1$ contributes to no pattern of the form $32-1$. Thus, any occurrence of that pattern must come from $\sigma_2\sigma_3\dotsm\sigma_n$. But as $\sigma$ is a fixed point $\sigma'$ must also be a fixed point, so $\sigma_2\sigma_3\dotsm\sigma_n$ contributes 0 to the statistic.

  Lastly, we show that over any orbit of size $2$, the change in the statistic is $\pm 1$. For simplicity,  say $(jk,l)$ contributes to the statistic if they have the pattern $32-1$, that is if the values $j,k,$ and $l$ occur in that order in the permutation and have $j > k > l$ with $k$ following directly after $j$. For example, $(32,1)$ contributes to the statistic for $\sigma = 3241$ and $(42,1)$  in $\sigma=4231$.

  For $n \leq 2$, all orbits are of size $1$, so there is nothing to show. For $n = 3$, the only orbit of size $2$ is $\{321, 312\}$, which we can verify has a change in the statistic of $\pm 1$. For $\sigma$ in an orbit of size $2$, either $2$ and $1$ are swapped in $\sigma$ or  either $\sigma_1 =1$ or $\sigma_1=2$ and $1$ and $2$ are swapped in some $\sigma', \sigma'', \dotsc$ So, we show for $n \geq 3$ that swapping $1$ and $2$ in $\sigma$ contributes $\pm 1$ to the statistic and then use induction for the cases $\sigma_1=1$ and $\sigma_1=2$, as $\sigma_1$ does not contribute to any occurrences of the pattern.

  If $(j2,1)$ contributes to the statistic, then swapping $1$ and $2$ changes the order of $1$ and $2$ and thus removes one occurrence of the pattern $32-1$. As $2$ moves to the right, any triple of the form $(jk, 2)$ still contributes to the pattern in $\psi(\sigma)$, and any new triples $(jk, 2)$ contributing to the statistic in $\psi(\sigma)$ would have been contributing as $(jk, 1)$ in $\sigma$. If $(jk, 1)$ contributes to the statistic in $\sigma$, then it either remains in $\psi(\sigma)$ (if $1$ is still to the right of $k$ after the swap) or it becomes $(jk, 2)$ in $\psi(\sigma)$ (if $1$ is to the left of $k$ after the swap). There can be no new contributions of the form $(jk, 1)$ in $\psi(\sigma)$ as $1$ moved to the left. Thus, the change in the total number of occurrences of the pattern $32-1$ is $-1$.

  If instead $1$ is to the left of $2$ in the statistic, then switching $1$ and $2$ adds $(j2, 1)$ as an occurrence of the pattern $32-1$ in $\psi(\sigma)$. Any triple of the form $(jk, 2)$ in $\sigma$ either remains in $\psi(\sigma)$ (if $2$ is still to the right of $k$ after the swap) or becomes $(jk, 1)$ in $\psi(\sigma)$ (if $2$ is to the left of $k$ after the swap). There can be no new contributions of the form $(jk, 2)$ in $\psi(\sigma)$ as $2$ moved to the left. Any triple of the form $(jk, 1)$ in $\sigma$ remains in $\psi(\sigma)$ as $1$ moved to the right, and any new triples $(jk, 1)$ contributing to the statistic in $\psi(\sigma)$ would have been contributing as $(jk, 2)$ in $\sigma$. Thus, the change in the total number of occurrences of the pattern $32-1$ is $1$.
\end{proof}

\begin{example}
  We give an example of $\psi$ from the proof of Theorem \ref{thm:357}. Let $n=4$ and $\sigma = 1432$. Since $\sigma_1=1$ we must look at $\sigma' = 321$ a permutation of $[3]$. $\psi(321) = 312$, so $\psi(1432) = 1423$. The number of occurrences of the pattern $32-1$ is $1$ for $\sigma$ and 0 for $\psi(\sigma)$.

  On the other hand, if we consider $\sigma = 1342$, then $\sigma' = 231$. Then $\psi(\sigma') = 231$, so $\psi(\sigma) = 1342$, and $\sigma$ is a fixed point.
\end{example}

\begin{cor}\label{cor:357} The number of occurrences of the pattern $12-3$ (Statistic $357$) exhibits the CSP under involutions with $2^{n-1}$ fixed points.
\end{cor}

\begin{proof}
    Each instance of a $32-1$ pattern invertibly transforms to a $12-3$ pattern by the complement map, thus the two statistics are equidistributed and the CSP holds for the number of occurrences of the pattern $12-3$ also.
\end{proof}

\subsection{Conjectured instances of the CSP}
We leave the following as conjectures. They have been tested up to $n=10$.

\begin{conj}\label{conj:123}
  The difference in Coxeter length of a permutation and its image under the Simion-Schmidt map (Statistic $123$) exhibits the CSP under involutions with $2^{n-1}$ fixed points.
\end{conj}

\begin{conj}\label{conj:373}
  The number of weak excedances that are also mid-points of a decreasing subsequence of length $3$ (Statistic $373$) is equidistributed with the cycle descent number (Statistic $317$) and thus exhibits the CSP under involutions with $2^{n-1}$ fixed points.
\end{conj}

\begin{remark}
After the initial preprint version of this paper was posted, Bishal Deb proved both of these conjectures. The proof of the latter involves continued fractions. These will appear in forthcoming work.
\end{remark}


\section{Involutions with $2^{\lfloor\frac{n}{2}\rfloor}$ fixed points}
\label{sec:alex-k}
In this section, we prove instances of the CSP for involutions that have $2^{\lfloor\frac{n}{2}\rfloor}$ fixed points. We first give an example of such a map in Subsection \ref{subsec:ak} and  then prove CSP results for two  statistics related to partial extrema in Subsection \ref{subsec:ak_l2r} and another related to inversions in Subsection \ref{subsec:ak_inv}.

\subsection{The Alexandersson-Kebede map}\label{subsec:ak}
The Alexandersson-Kebede map $\kappa$ is an involution on permutations that preserves right-to-left minima. It was introduced in \cite{AlexanderssonKebede} to give a bijective proof of an identity regarding derangements and to refine it with respect to right-to-left minima. We give the definition, important properties, and an example below.

\begin{definition}
  Let $\sigma\in S_n$. The \textbf{Alexandersson-Kebede map} $\kappa$ is defined as follows. Let $i$ be the smallest odd integer such that $\sigma\cdot(i, i+1)$ and $\sigma$ have the same set of right-to-left minima, if such an $i$ exists. If it does not, then $\kappa(\sigma) = \sigma$. Otherwise, $\kappa(\sigma) = \sigma\cdot (i, i+1)$.
\end{definition}

\begin{prop}[Lemma 4.1.3 of \cite{AlexanderssonKebede}]\label{prop:AK_properties}
  The map $\kappa$ satisfies the following:
  \begin{itemize}
    \item $\kappa$ is an involution.
    \item $\kappa$ preserves the number of right-to-left minima.
    \item The fixed points of $\kappa$ are permutations satisfying $\{\sigma(i), \sigma(i+1)\} =\{i, i+1\}$ for all odd $i$. These are called \textit{decisive permutations} in \cite{AlexanderssonKebede}.
    \item There are  $2^{\lfloor \frac{n}{2}\rfloor}$ fixed points.
  \end{itemize}
\end{prop}

\begin{example}
  Consider the permutation $\sigma = 2134756$. Its right-to-left minima are $\{1,3,4,5,6\}$. To find $\kappa(\sigma)$, we need to find the smallest odd $i$ for which $\sigma\cdot(i, i+1)$ and $\sigma$ have the same right-to-left minima. We cannot choose $i=1$ since $\sigma\cdot(1,2) = 1234756$ makes $2$ a right-to-left minimum. We cannot choose $i=3$ since  $\sigma\cdot(3,4) = 2143756$ makes $4$ not a right-to-left minimum anymore. Then $i=5$ since $2134756$ and $2134576$ have the same set of right-to-left minima. Here, $\kappa(\sigma) = 2134576$.
\end{example}

\subsection{The sum of the numbers of left-to-right maxima and right-to-left minima}
\label{subsec:ak_l2r}
In this subsection, we use the Alexandersson-Kebede map $\kappa$ to show that some statistic generating functions related to partial extrema exhibit the CSP with respect to involutions with $2^{\lfloor\frac{n}{2}\rfloor}$ fixed points. 
We first need two lemmas.

\begin{lem}\label{lem:L2Rmax-non_fixed_points}
  When $\sigma\in S_n$ is not fixed by $\kappa,$ the number of left-to-right maxima for $\sigma$ and for $\kappa(\sigma)$ differ by one.
\end{lem}

\begin{proof}
  Let $\sigma\in S_n$ so that $\kappa$ does not fix $\sigma.$ Then there exists a smallest odd integer $i$ so that the right-to-left minima of $\sigma$ are equal to the right-to-left minima of $\kappa(\sigma).$ Note that for all odd integers $j<i,$ the sets of right-to-left minima for $\sigma(j~j+1)$ and $\sigma$ are different. Thus, either $\sigma(j),\sigma(j+1),$ or both are right-to-left minima for all odd integers $j<i,$ which means $\sigma(i)$ and $\sigma(i+1)$ are greater than $\sigma(j)$ and $\sigma(j+1)$ for all odd $j<i.$
  Then either $\sigma(i)>\sigma(i+1)$ or $\sigma(i)<\sigma(i+1).$ If $\sigma(i)>\sigma(i+1),$ then only $\sigma(i)$ is a left-to-right maximum. If $\sigma(i)<\sigma(i+1),$ then both $\sigma(i)$ and $\sigma(i+1)$ are left-to-right maxima. So, under the action of $\kappa$ we have $\sigma(i~i+1)$, meaning that over one orbit, the number of left-to-right maxima differs by one.
\end{proof}

\begin{lem}\label{lem:R2Lmin=L2Rmax_fixed_points}
  The number of left-to-right maxima and the number of right-to-left minima are equal for permutations fixed under $\kappa$.
\end{lem}

\begin{proof}
  Recall from Proposition \ref{prop:AK_properties} that the fixed points of $\kappa$ are permutations satisfying $\{\sigma(i), \sigma(i+1)\} =\{i, i+1\}$ for all odd $i$. We prove that, for each pair $\{i, i+1\}$ with odd $i$, either only one of them is a left-to-right maximum and the other one is a right-to-left minimum, or that both of them are left-to-right maxima and right-to-left minima.

  For a permutation $\sigma$ fixed under $\kappa$ and a pair $\{i, i+1\}$ with odd $i$, we know that all entries to the left of $i$ and $i+1$ are smaller, and that all entries to the right of $i$ and $i+1$ are larger. Hence, there is at least one left-to-right maximum and one right-to-left minimum among $\{i, i+1\}$.

  If $\sigma(i) = i$ (and $\sigma(i+1) = i+1$), there are two left-to-right maxima and two right-to-left minima. If $\sigma(i+1) = i$ (and $\sigma(i) = i+1$), then only $\sigma(i+1)$ is a right-to-left minimum, and $\sigma(i)$ is a left-to-right maximum. Either way, the number of left-to-right maxima is the same as the number of right-to-left minima for permutations that are fixed under $\kappa$.
\end{proof}

Using the two properties above, we prove the CSP for two statistics with respect to the Alexandersson-Kebede map, only one of which appears in FindStat. From it, it follows that the CSP occurs for any map with the same orbit structure.

\begin{thm}\label{thm:CSP_AK}
  The statistic given as the sum of the numbers of left-to-right maxima and right-to-left minima exhibits the cyclic sieving phenomenon with respect to any involution with $2^{\lfloor\frac{n}{2}\rfloor}$ fixed points.
\end{thm}

\begin{proof}
  Since $\kappa$ is an involution, we only need to show that the statistic generating function evaluated at $q=-1$ gives the number of fixed points. By Proposition \ref{prop:AK_properties}, this is $2^{\lfloor \frac{n}{2}\rfloor}$.

  Let $\Stat$ be the statistic counting the sum of the numbers of right-to-left minima and left-to-right maxima. Then, we need to show that
  \[\sum_{\sigma \in S_n} (-1)^{\Stat(\sigma)} = 2^{\lfloor \frac{n}{2}\rfloor}.\]

  We break the sum above into two parts, to account for the fixed points separately:
  \begin{equation}
    \sum_{\sigma \in S_n} (-1)^{\Stat(\sigma)} = \sum_{\substack{\sigma \in S_n\\ \kappa(\sigma) = \sigma}} (-1)^{\Stat(\sigma)} + \sum_{\substack{\sigma \in S_n\\ \kappa(\sigma) \neq \sigma}} (-1)^{\Stat(\sigma)}.\label{eq:sum_CSP_AK}
  \end{equation}

  We know from Proposition \ref{prop:AK_properties} and Lemma \ref{lem:L2Rmax-non_fixed_points} that
  \[\sum_{\substack{\sigma \in S_n\\ \kappa(\sigma) \neq \sigma}} (-1)^{\Stat(\sigma)} =\sum_{\text{orbits of }\kappa} ((-1) + 1) = 0.\]

  Furthermore, when we write $\text{L2R}$ and $\text{R2L}$ for left-to-right and right-to-left, respectively, Lemma \ref{lem:R2Lmin=L2Rmax_fixed_points} tells us that
  \[\sum_{\substack{\sigma \in S_n\\ \kappa(\sigma) = \sigma}} (-1)^{\#\text{L2R maxima}(\sigma) + \#\text{R2L minima}(\sigma)}=\sum_{\substack{\sigma \in S_n\\ \kappa(\sigma) = \sigma}} \underbrace{(-1)^{2\#\text{L2R maxima}(\sigma)}}_{1} = \#\{\sigma \in S_n\mid \sigma = \kappa(\sigma)\} = 2^{\lfloor \frac{n}{2}\rfloor}.\]

  Hence, Equation \eqref{eq:sum_CSP_AK} becomes
  \[\sum_{\substack{\sigma \in S_n\\ \kappa(\sigma) = \sigma}} (-1)^{\Stat(\sigma)} + \sum_{\substack{\sigma \in S_n\\ \kappa(\sigma) \neq \sigma}} (-1)^{\Stat(\sigma)} =  2^{\lfloor\frac{n}{2}\rfloor} + 0,\]
  which is exactly the number of fixed points under $\kappa$. This shows that the map exhibits the cyclic sieving phenomenon for the sum of the numbers of left-to-right maxima and right-to-left minima. Therefore, any involution with $2^{\lfloor\frac{n}{2}\rfloor}$ fixed point exhibits the CSP for that statistic.
\end{proof}

\begin{cor}
  \label{cor:1005}
  The number of indices of a permutation that are either left-to-right maxima or right-to-left minima but not both (Statistic $1005$) exhibits the cyclic sieving phenomenon with respect to any involution with $2^{\lfloor\frac{n}{2}\rfloor}$ fixed points.
\end{cor}

\begin{proof}
  This follows from Theorem \ref{thm:CSP_AK}. Since the map is an involution, we prove the cyclic sieving phenomenon by evaluating the statistic generating function at $q=-1$. This means that only the parity of the statistic matters. To calculate the number of indices for a permutation that are either left-to-right maxima or right-to-left minima but not both, we subtract an even number from the sum of the numbers of left-to-right maxima and right-to-left minima, meaning that the statistics, for each permutation, always have the same parity. Thus, the cyclic sieving phenomenon also holds for the number of indices for a permutation that are either left-to-right maxima or right-to-left minima but not both with respect to any involution with $2^{\lfloor\frac{n}{2}\rfloor}$ fixed points.
\end{proof}

\subsection{The number of invisible inversions} 
\label{subsec:ak_inv}
In this subsection, we prove the CSP for the number of invisible inversions (Statistic 1727)  with respect to involutions with $2^{\lfloor\frac{n}{2}\rfloor}$ fixed points.

\begin{definition}
  An inversion $(i,j)$ of $\sigma$ is said to be \emph{invisible} if it satisfies $\sigma(i)>\sigma(j)>i$.
\end{definition}

\begin{thm}\label{thm:CSP_AK-Invisible}
  The number of invisible inversions (Statistic $1727$) exhibits the cyclic sieving phenomenon with respect to involutions with $2^{\lfloor\frac{n}{2}\rfloor}$ fixed points.
\end{thm}

\begin{proof}
  To show that the number of invisible inversions exhibits the CSP, we define a new involution $\psi: S_n \rightarrow S_n$ with $2^{\lfloor\frac{n}{2} \rfloor}$ fixed points; see Example~\ref{ex:kappa_psi}. We show the fixed points of $\psi$ have no invisible inversions, and orbits of size $2$ are made of two permutations whose number of invisible inversions differs by $1$. 

Define $\psi: S_n \rightarrow S_n$ as follows.

\begin{itemize}

    \item If $n$ is even, section the permutation into blocks of size $2$ by grouping the spots $i$ and $i + 1$ with $i$ odd. If there is some pair of values $(i, i+1)$ with $i$ odd where $i$ and $i + 1$ are not both in the $\frac{i+1}{2}$-th block, pick $i$ to be maximal and define $\psi(\sigma)$ to be the involution swapping the values $i$ and $i + 1$.

    \item If $n$ is odd, place $1$ in its own block and section the remainder of the permutation into blocks of size $2$ by grouping spots $i$ and $i + 1$ with $i$ even.  If there is some pair of values $(i, i+1)$ with $i$ even where $i$ and $i + 1$ are not both in the $\left(\frac{i}{2} + 1\right)$-th block, pick $i$ to be maximal and define $\psi(\sigma)$ by swapping the values $i$ and $i + 1$.

\end{itemize}
The fixed points are the permutations of $[n]$ where $i$ and $i+1$ are both in the $\lceil \frac{i+1}{2}\rceil$-th block for all $i$, where $i$ and $n$ have different parities. The number of such permutations is $2^{\lfloor\frac{n}{2} \rfloor}$, since ${\lfloor\frac{n}{2} \rfloor}$ is the number of blocks of size $2$.
  
  Let $\sigma$ be a permutation not fixed by $\psi$ and let $(i, i + 1)$ be the swapped pair. We show that swapping $i$ and $i + 1$ either breaks or creates an invisible inversion with those values. Consider $j, k$ such that $1 \leq j, k \leq n$, $\sigma_j = i$, $\sigma_k = i + 1$. Then $\{j , k\} \neq \{i, i+1 \}$. As $i$ is chosen to be maximal, we have $j, k \leq i+ 1$ since every value greater than $i + 1$ must already be in its assigned block to the right of spot $i + 1$. Similarly, for all $t \leq i + 1$, we have $\sigma_t \leq i + 1$.

  If $j < k$, then $\sigma_ j = i$ cannot be in its block, so $j < i$. Thus, $\psi(\sigma)_j = \sigma_k = i + 1 > \psi(\sigma)_k = \sigma_j = i > j,$ so $(j, k)$ is an invisible inversion of $\psi(\sigma)$ but not in $\sigma$ because $\sigma_j = i < \sigma_k = i + 1$. If $k < j$, then $\sigma_k = i + 1$ cannot be in its block, so $k < i$. Thus, $\sigma_k = i + 1 > \sigma_j = i > k$, so $(k , j)$ is an invisible inversion of $\sigma$ but not in $\psi(\sigma)$ since $\psi(\sigma)_k = \sigma_j = i < \psi(\sigma)_j = \sigma_k = i +1$.

  Additionally, swapping these values does not create or break any other invisible inversions. First, note that any invisible inversion not involving spots $j$ and $k$ is unaffected by this map, so we only need to consider invisible inversions involving either spot $j$ or spot $k$.  Let $\sigma_t \notin \{i, i+1\}$. Recall that if $t \leq i + 1$, then $\sigma_t \leq i + 1$, so since $\sigma_t \notin \{i, i + 1\}$, if $t \leq i + 1$, then $\sigma_t < i$.
  Hence there are no inversions of the form $(t, j)$ or $(t, k)$. Thus, there are two cases to check.

  \begin{itemize}

    \item If $(j, t)$ is an invisible inversion in $\sigma$, then $j < t$ and $\sigma_j = i > \sigma_t > j$. Then $\psi(\sigma)_j = i + 1 > \sigma_t > j$ and $(j, t)$ is an invisible inversion in $\psi(\sigma)$. Similarly, if $(j, t)$ is an invisible inversion in $\psi(\sigma)$, then $\psi(\sigma)_j = i + 1 > \psi(\sigma)_t = \sigma_t > j$ and as $\sigma_t \neq i$, $\sigma_j = i > \sigma_t > j$ and $(j, t)$ is an invisible inversion in $\sigma$.

    \item If $(k, t)$ is an invisible inversion in $\sigma$, then $k < t$ and $\sigma_k = i + 1 > \sigma_t > k$. As $\sigma_t \neq i$, then $\psi(\sigma)_k = i > \sigma_t > k$ and $(k, t)$ is an invisible inversion in $\psi(\sigma)$. Similarly, if $(k, t)$ is an invisible inversion in $\psi(\sigma)$, then $\psi(\sigma)_k = \sigma_j = i  > \psi(\sigma)_t = \sigma_t > k$, so $\sigma_k = i + 1 > \sigma_t > k$ and $(k, t)$ is an invisible inversion in $\sigma$.

\end{itemize}
Therefore, the number of invisible inversions in $\sigma$ and $\psi(\sigma)$ differs by $1$ when $\sigma$ is not a fixed point of $\psi$. Since fixed points of $\psi$ have no invisible inversions, the argument used in the proof of Theorem \ref{thm:CSP_AK} is used to prove that involutions with $2^{\lfloor \frac{n}{2}\rfloor}$ fixed points exhibit the CSP for the number of invisible inversions.
\end{proof}

\begin{example}
  \label{ex:kappa_psi}
  Let $n = 8$. Then $\sigma = 21534687$ maps to $\psi(\sigma) = 21634587$.  Since $7$ and $8$ are already in the correct block, we switched $5$ and $6$. There is one invisible inversion in $\sigma$, $(3,5)$, and two invisible inversions in $\psi(\sigma)$, $(3, 5)$ and $(3, 6)$. In this case, $21345687$ is a fixed point.

  Next, let $n = 9$. Then $\sigma = 215346879$ maps to $\psi(\sigma) = 215346978$.  There is one invisible inversion in $\sigma$, $(3, 5)$, and two invisible inversions in $\psi(\sigma)$, $(3, 5)$ and $(7, 9)$. In this case, $215346879$ is not a fixed point, but $132456789$ is a fixed point.

\end{example}

\section{Involutions without fixed points}

\label{sec:rev_comp}

In this section, we prove instances of CSP for involutions that have no fixed points, that is, all bijections on permutations that only produce orbits of size $2$.
This includes the reverse and complement maps (Definition \ref{def:maps}). As noted in Subsection \ref{sec:bg_CSP}, studying involutions means that the results in this section can also be characterized as instances of the $q=-1$ phenomenon. So for involutions without fixed points and statistics where we have an explicit generating function $f(q)$, we need only check that $f(-1)=0$ to verify the CSP.

We organize our results based on proof technique; the first subsection uses the generating function for the statistic, while the second uses bijective proofs that pair the permutations to show the number whose statistics value is odd  equals the number whose statistic value is even.

\subsection{Factoring the generating functions}
We begin with results that make use of an explicit formula for the generating function of a statistic. The main result of this subsection is \Cref{circled_shifted}, related to the number of circled entries of the shifted recording tableau, but we also include other easier results for the sake of example and comprehensiveness.

We begin with the following proposition and its corollaries. 

\begin{prop}\label{thm:rev_comp_cyc_decomp}
  For $n \geq 2$, 
  the number of cycles in the cycle decomposition (Statistic $31$) exhibits the cyclic sieving phenomenon under involutions without fixed points.
\end{prop}

The proof of Proposition \ref{thm:rev_comp_cyc_decomp} is clear, as \cite{Stanley2011} gives the generating fuction $f_n(q) = q \prod_{k=1}^{n-1} (q+k)$ so that $f_n(-1)=0$. 
We then have the following corollary, since the number of cycles is equidistributed with left-to-right maxima~\cite[Prop.~1.3.1]{Stanley2011} and the other statistics by symmetry.

\begin{cor}\label{cor:stat31_etal}
  The following statistics are equidistributed with the number of cycles 
  and thus also exhibit the cyclic sieving phenomenon for $n\geq 2$ with respect to involutions without fixed points.
  \begin{itemize}
    \item Statistic $7$: The number of saliances (right-to-left maxima),
    \item Statistic $314$: The number of left-to-right-maxima,
    \item Statistic $542$: The number of left-to-right-minima,
    \item Statistic $991$: The number of right-to-left minima.
  \end{itemize}
\end{cor}

\begin{cor}\label{cor:rev_and_comp_541}
  For $n \geq 2$, the number of indices greater than or equal to $2$ such that all smaller indices appear to its right (Statistic $541$) exhibits the cyclic sieving phenomenon under involutions without fixed points.
\end{cor}

\begin{proof}
  Note that for $n\geq 2$, the generating function $f(q) = \frac{1}{q} f_n(q)$, where $f_n(q)$ is defined in the proof of  Proposition \ref{thm:rev_comp_cyc_decomp}. This is because the number of indices greater than or equal to 2 of a permutation such that all smaller indices appear to its right equals the number of left-to-right minima minus one.
  Thus $f(-1)=0$. 
\end{proof}

\begin{definition}\label{def:abs_len}
  The \textbf{absolute length} of a permutation of $[n]$ equals
  $n$ minus the number of cycles.
\end{definition}

\begin{prop} \label{thm:rev_comp_abs_len}
  For $n \geq 2$, the following statistics exhibit the cyclic sieving phenomenon under involutions without fixed points.
  \begin{itemize}
    \item Statistic $216$: The absolute length,
    \item Statistic $316$: The number of non-left-to-right-maxima.
  \end{itemize}
\end{prop}

The proof of Proposition \ref{thm:rev_comp_abs_len} follows similarly to that of Proposition \ref{thm:rev_comp_cyc_decomp}, as the number of cycles  is equidistributed with the number of left-to-right minima. Combining this with Definition \ref{def:abs_len} gives that absolute length has 
generating function $q^n f_n(q^{-1})$ (as does the number of non-left-to-right minima).

The final result of this subsection relies on work of Sagan \cite{Sagan1987} and Schur \cite{Schur1911} on generating functions related to shifted tableau.

\begin{thm}\label{circled_shifted}
  For $n \geq 2$, the number of circled entries of the shifted recording tableau \textnormal{(Statistic 864)} exhibits the cyclic sieving phenomenon under involutions without fixed points.
\end{thm}
\begin{proof}
  A shifted partition shape of $n$ is a list of strictly decreasing numbers summing to $n$; the diagram of a shifted partition is drawn with each row indented one more box than the previous. There is a bijection (analogous to the Robinson-Schensted correspondence) between permutations and pairs of shifted standard tableau $(P,Q)$ of the same shape, where $Q$ has a subset of its non-diagonal entries circled \cite[Theorem 3.1]{Sagan1987}. Since this is a bijection, all possibilities of circled entries appear. The bijection gives a combinatorial proof of the following identity \cite[Corollary 3.2]{Sagan1987} (originally due to Schur \cite{Schur1911})
  \[n!=\sum_{\lambda\models n}2^{n-\ell(\lambda)}g_{\lambda}^2,\]
  where the sum is over all shifted partition shapes $\lambda$ of $n$, $\ell(\lambda)$ is the number of rows of $\lambda$ (so that $n-\ell(\lambda)$ is the number of non-diagonal entries), and $g_{\lambda}$ is the number of shifted standard tableaux of shape $\lambda$.
  The generating function for the number of circled entries is
  $f(q)=\sum_{\lambda\models n}(1+q)^{n-\ell(\lambda)}g_{\lambda}^2$,
  and we obtain $f(-1)=0$ as desired.
\end{proof}

\subsection{Bijective Proofs}
In this subsection, we consider results related to statistic generating functions which do not have known, or easily discovered, formulas. Thus the proofs are more difficult than simply plugging in $-1$ to a known generating function. We prove these results by pairing the permutations so that the statistic changes in parity across the pair. First, we include those we paired with either the reverse or the complement map. Then we provide other involutions without fixed points for the remaining results.

The main results in this subsection are Theorems \ref{inv_at_most_2}, \ref{inc_dec}, \ref{thm:bi_alt}, and \ref{conj:patterns}, the last of which relates to the number of occurrences of certain patterns of length three. In addition, we have \Cref{strict_3_des}, which relates to statistics which exhibit the CSP for certain values of $n$. The section also contains results such as \Cref{cyclically}, which is related to the number of cyclically simple transpositions needed to sort a permutation, among others. We conclude the section with \Cref{conj:inv_distance_3}.

We will begin with \Cref{inv_at_most_2}, which relates to the following definition.

\begin{definition}\label{def:inv_k}
  Let $\sigma=\sigma_1 \sigma_2\ldots \sigma_n \in S_n$ for $n\geq 2$. An \textbf{inversion of distance $k$} is a pair $(\sigma_i, \sigma_{i+k})$ such that $\sigma_i>\sigma_{i+k}$. In this context, a descent is called an inversion of distance one.
  Let $\inv_k(\sigma)$ denote the number of inversions of distance $k$ in $\sigma$.
\end{definition}

\begin{thm}\label{inv_at_most_2}
  For $n \geq 2$, the number of inversions of distance at most $2$ (Statistic $495$) exhibits the cyclic sieving phenomenon under involutions without fixed points.
\end{thm}

\begin{proof}
  Let $n\geq 2$
  and $\sigma \in S_n$, and recall that $\mathcal{C}(\sigma)$ denotes the complement of $\sigma$ (\Cref{def:maps}). Let $\inv_2(\sigma)$ denote the number of inversions of distance at most 2 for $\sigma$, as seen in Definition \ref{def:inv_k}. From \cite[Proposition 5.28]{ELMSW22}, if the pair $(i, j)$ is an inversion pair of $\sigma$ of distance at most 2, then $(i,j)$ is not an inversion pair for $C(\sigma)$ of distance at most 2.

  It is also shown in \cite[Proposition 5.28]{ELMSW22} that there are $2n-3$ possible inversion pairs of distance at most two for any permutation. So, if $\sigma$ has $p$ inversion pairs of distance at most 2, then $C(\sigma)$ has $(2n-2) - p$ inversion pairs of distance at most $2$.

  We note that the statistic generating function can be written as
  \[
    f(q) = \sum_{\sigma \in S_n} q^{\inv_2(\sigma)} = \sum_{p=0}^{2n-2} a_p q^p,
  \]
  for $n\geq 2$.
From \cite[Proposition 5.16]{ELMSW22}, we note that $\inv(\sigma)+\inv (\C(\sigma)) = \frac{n(n-1)}{2}$, which is the maximum number of inversion pairs in a permutation of length $n$. Hence, the coefficients $a_{p}$ and $a_{2n-3-p}$ are equal.
  For any $n$ value, the parity of $p$ and $2n-3-p$ are opposite. Thus, $a_{p}(-1)^p + a_{2n-3-p}(-1)^{2n-3-p} = 0$, which gives us $f(-1) = 0$ as desired.
\end{proof}

Proposition \ref{thm:inv_distance_3} can be shown using a similar proof as in Theorem \ref{inv_at_most_2}, but it has only been proven for $n$ odd. However, we believe it also holds for $n$ even (see Conjecture \ref{conj:inv_distance_3}).

\begin{prop}\label{thm:inv_distance_3}
  For $n \geq 3$ odd, the number of inversions of distance at most $3$ (Statistic $494$) exhibits the cyclic sieving phenomenon under involutions without fixed points.
\end{prop}

  The following definition is used in Theorem \ref{strict_3_des}.
  \begin{definition}\label{def:width_k}
  Recall that $\Des(\sigma)$ is the set of indices where descents occur, and $\des(\sigma) = |\Des(\sigma)|$. Similarly, let $\Des_k(\sigma) = \{ i\mid\sigma_i>\sigma_{i+k}\}$ be the set of indices where \textbf{width $k$-descents} occur (see \cite{widthk}). Let $\des_k(\sigma) = |\Des_k(\sigma)|$. 
\end{definition}
  
\begin{remark}
We remark that FindStat defines these descents using the phrasing ``descents of distance 2 of a permutation'' and ``strict 3-descents of a permutation,'' and the reader may also encounter the notion of $k$-descents elsewhere in the literature. We have standardized Definition \ref{def:width_k} to ``width $k$-descents'' and will utilize this notion throughout the following results.
\end{remark}

The following result shows that the CSP holds for width $k$-descents when $n$ and $k$ have opposite parity. 
\begin{thm}\label{strict_3_des}
  For $k \geq 1$ and $n\geq 2$, the number of width $k$-descents exhibits the cyclic sieving phenomenon under involutions without fixed points when $n$ and $k$ have the opposite parity. In particular, this includes the following FindStat statistics in the indicated cases:
  \begin{itemize}
    \item Statistic $21$: The number of descents for $n\geq 2$, when $n$ is an even number.
    \item Statistic $836$: The number of width $2$-descents, when $n\geq 3$ and $n$ is an odd number.
    \item Statistic $1520$: The number of width $3$-descents, when $n\geq 4$ and $n$ is an even number.
  \end{itemize}
\end{thm}

\begin{proof}
  Let $n\geq 2$ and $\sigma \in S_n$, and let $\mathcal{C}(\sigma)  $ be the complement of $\sigma$. 
  Note that any descent in $\sigma$ is mapped to an ascent in $C(\sigma)$ and that there are $n-k$ possible width $k$-descents (see e.g.~\cite[Prop 5.26-5.27]{ELMSW22}). 
  Thus if $\sigma$ has $d$ width $k$-descents, then $C(\sigma)$ has $n-k - d$ width $k$-descents.
  Thus, the generating function
$f(q) = \sum_{\sigma \in S_n} q^{\des_k(\sigma)} = \sum_{p=0}^{n-k}a_pq^p$
  is a palindromic polynomial.
  If $n$ and $k$ have opposite parity, then $n-k$ is odd. So, $d$ and $n-k-d$ have the opposite parity, $a_d(-1)^d +a_{n-k-d}(-1)^{n-k-d} = 0$, and $f(-1)=0$.
\end{proof}

\begin{remark}
After the initial preprint version of this paper was posted, Sergi Elizalde found by a different proof method the precise condition on $n$ and $k$ under which the CSP for width $k$-descents holds, namely, if and only if $n$ is not congruent to $k$ modulo $2k$. 
\end{remark}

For the remaining results in this subsection, we use a map other than the reverse or the complement to pair the permutations. We define a transposition $\psi$ on the set of permutations of $[n]$ and show the parity of the statistic on $\sigma$ and $\psi(\sigma)$ differs. We provide a detailed proof for Theorem \ref{inc_dec} below, but for other results, we simply give the map used, unless more details are necessary. (Full details are included in version 1 of the arXiv preprint.)

\begin{thm}\label{inc_dec}
  For $n \geq 4$, the number of times a permutation switches from increasing to decreasing or from decreasing to increasing (Statistic $483$) exhibits the cyclic sieving phenomenon under involutions without fixed points.
\end{thm}

\begin{proof}
  We see the CSP for this statistic does not hold for $n = 3$, as $1 2 3$ and $3 2 1$ never change from increasing to decreasing, but all other permutations change once. So the generating function is given by $f(q) = 2 + 4q$ and $f(-1) = -2$.

  Let $\sigma = \sigma_1\sigma_2 \dotsm \sigma_n\in S_n$ with $n \geq 4$. Define the transposition $\psi(\sigma) = \sigma_1\sigma_2 \dotsm \sigma_n \sigma_{n-1}$. Swapping $\sigma_{n-1}$ and $\sigma_{n}$ does not affect any changes in increasing and decreasing prior to $\sigma_{n-2}$. If $\sigma_{n-2}\sigma_{n-1}\sigma_n$ contains the pattern $123$ (see Definition~\ref{def:patterns}), then this contributes no changes to increasing or decreasing in $\sigma$, but in $\psi(\sigma)$, it becomes the pattern $132$, which contributes $+1$ to the statistic. Similarly, if $\sigma$ has the pattern $132$, then $\psi(\sigma)$ has the pattern $123$ and the statistic changes by 1 between $\sigma$ and $\psi(\sigma)$. If $\sigma_{n-2}\sigma_{n-1}\sigma_n$ has the pattern $321$, then this contributes no changes to increasing or decreasing in $\sigma$, but in $\psi(\sigma)$, it becomes the pattern $312$, which contributes +1 to the statistic. Similarly, if $\sigma$ has the pattern $312$, then $\psi(\sigma)$ has the pattern $321$ and the statistic changes by 1 between $\sigma$ and $\psi(\sigma)$.

  If $\sigma_{n-2}\sigma_{n-1}\sigma_n$ contains the pattern $213$, then we consider two cases. If $\sigma_{n-3} < \sigma_{n-2}$, then $\sigma$ changes between increasing and decreasing at $\sigma_{n-2}$ and $\sigma_{n-1}$, but $\psi(\sigma)$ loses the change at the $n-2$ position as it ends with the pattern $231$. If $\sigma_{n-3} > \sigma_{n-2}$, then $\sigma$ changes between increasing and decreasing at $\sigma_{n-1}$, but $\psi(\sigma)$ adds a change at the position $n-2$ as it ends with the pattern $231$. Similarly, if $\sigma$ has the pattern $231$, then $\psi(\sigma)$ has the pattern $213$ and the statistic changes by $1$ between $\sigma$ and $\psi(\sigma)$.

  Thus, the parity in the number of times a permutation switches from increasing to decreasing differs between $\sigma$ and $\psi(\sigma)$.
\end{proof}

In Theorem \ref{Mahonian}, we will show that the number of inversions of a permutation exhibits the cyclic sieving phenomenon under both the reverse and complement maps, as it is a Mahonian statistic. For now, instead of considering all inversions, we restrict to even inversions.

\begin{definition} An inversion is an \textbf{even inversion} when the indices $i$ and $j$ have the same parity.
\end{definition}

\begin{prop}\label{ev_inv}
  For $n \geq 3$, the number of even inversions (Statistic $538$) exhibits the cyclic sieving phenomenon under involutions without fixed points.
\end{prop}

This result does not hold for $n = 2$ since the generating function is $f(q) = 2$. For $n \geq 3$, Proposition \ref{ev_inv} is proved using $\psi(\sigma) = \sigma_3\sigma_2\sigma_1 \dotsm \sigma_n$.
%

Odd inversions are defined as inversions where the indices $i$ and $j$ differ in parity. The number of odd inversions  (Statistic 539) does not exhibit an analogous CSP; one can verify from the generating functions for $n=4$ and $n=5$ that can be found on the FindStat page for the statistic.

\begin{definition}
  The \textbf{number of up-down runs} of a permutation $\sigma$ equals the number of maximal monotone consecutive subsequences of $\sigma$ plus 1 if $1$ is a descent ($\sigma_1 > \sigma_2)$.
\end{definition}

\begin{example}
  Let $\sigma = 53142$. Then, the number of up-down runs of $\sigma$ is $4$. These are $531, 14, 42,$ as well as the descent in the first position.
\end{example}

\begin{prop}\label{up-down_runs}
  For $n \geq 2$, the number of up-down runs (Statistic $638$) exhibits the cyclic sieving phenomenon under involutions without fixed points.
\end{prop}

Proposition \ref{up-down_runs} can be verified by hand for $n = 2, 3$. For $n \geq 4$, we use $\psi(\sigma) = \sigma_1\sigma_2 \dotsm \sigma_n \sigma_{n-1}$.

\smallskip
The following definition is used in Theorem~\ref{thm:bi_alt}. 
\begin{definition}
  The \textbf{standardized bi-alternating inversion number} of a permutation $\sigma$ is given by
  \[\textrm{stat677}(\sigma) = \frac{j(\sigma) + (\lfloor \frac{n}{2} \rfloor)^2}{2},
  \mbox{ where }
  j( \sigma) = \displaystyle \sum_{1 \leq y < x \leq n} (-1)^{y + x} \textrm{sign}(\sigma_x - \sigma_y).\]
\end{definition}

\begin{thm}\label{thm:bi_alt}
  For $n \geq 2$, the standardized bi-alternating inversion number (Statistic $677$) exhibits the cyclic sieving phenomenon under involutions without fixed points.
\end{thm}

\begin{proof}
  Let $\sigma = \sigma_1\sigma_2 \dotsc \sigma_n$ be a permutation of $[n]$, and define $\psi(\sigma) = \sigma_3\sigma_2\sigma_1 \sigma_4 \dotsc \sigma_n$. Consider $3 < x \leq n$. Then
  \[\textrm{sign}(\psi(\sigma)_x - \psi(\sigma)_3) = \textrm{sign}(\sigma_x - \sigma_1)\]
  and
  \[\textrm{sign}(\psi(\sigma)_x - \psi(\sigma)_1) = \textrm{sign}(\sigma_x - \sigma_3).\]
  So, any pair of indices $(3, x)$ contributes the same to $\sigma$ and $\psi(\sigma)$ when calculating the statistic.

  Next, consider the pairs of indices $(1, 2)$, $(2, 3)$, and $(1, 3)$:
  \[\textrm{sign}(\psi(\sigma)_2 - \psi(\sigma)_1) = \textrm{sign}(\sigma_2 - \sigma_3) = - \textrm{sign}(\sigma_3 - \sigma_2),\]
  \[\textrm{sign}(\psi(\sigma)_3 - \psi(\sigma)_2) = \textrm{sign}(\sigma_1 - \sigma_2) = - \textrm{sign}(\sigma_2 - \sigma_1),\]
  and
  \[\textrm{sign}(\psi(\sigma)_3 - \psi(\sigma)_1) = \textrm{sign}(\sigma_1 - \sigma_3) = - \textrm{sign}(\sigma_3 - \sigma_1).\]
  Then,
  \begin{align*}
    j(\psi(\sigma)) = & j(\sigma) - (-1)^{1 + 2}\textrm{sign}(\sigma_2 - \sigma_1) - (-1)^{1 + 3}\textrm{sign}(\sigma_3 - \sigma_1) - (-1)^{2 + 3}\textrm{sign}(\sigma_3 - \sigma_2)                         \\
    +                 & (-1)^{1 + 2}\textrm{sign}(\psi(\sigma)_2 - \psi(\sigma)_1) + (-1)^{1 + 3}\textrm{sign}(\psi(\sigma)_3 - \psi(\sigma)_1) + (-1)^{2 + 3}\textrm{sign}(\psi(\sigma)_3 - \psi(\sigma)_2) \\
    =                 & j(\sigma) + 2\big( \textrm{sign}(\sigma_2 - \sigma_1) +  \textrm{sign}(\sigma_3 - \sigma_2) -  \textrm{sign}(\sigma_3 - \sigma_1)\big).
  \end{align*}

  Since $\textrm{sign}$ is always $\pm 1$, so odd, adding / subtracting three of them  results in an odd number. Thus $j(\psi(\sigma)) = j(\sigma) + 2(2k + 1)$ and $\textrm{stat677}(\psi(\sigma)) = \textrm{stat677}(\sigma) + 2k + 1.$
\end{proof}

The following definition of reduced reflection length is used in Proposition ~\ref{red_ref_len}. It is given as a general definition for Coxeter groups. When restricting to permutations, one can also calculate this statistic  as twice the depth of the permutation minus the usual length. For the purpose of our proof, we use the more general definition.
\begin{definition}\label{red_ref}
  Let $W$ be a Coxeter group. If $T$ is the set of reflections of $W$ and $\ell_C(w)$ is the usual Coxeter length given by the number of simple reflections needed in a reduced form of $w$. Then the \textbf{reduced reflection length} of $w \in W$ is given by
  \[\textrm{min}\{r \in \mathbb{N} \ | \ w = t_1t_2\dotsm t_r, \ t_i \in T, \ \ell_C(w) = \sum_i^r \ell_C(t_i)\}.\]
\end{definition}

\begin{prop}\label{red_ref_len}
  For $n \geq 2$, the reduced reflection length of the permutation (Statistic $809$) exhibits the cyclic sieving phenomenon under involutions without fixed points.
\end{prop}

Proposition \ref{red_ref_len} is proved with $\psi(\sigma) = \sigma_1\sigma_2\dotsm \sigma_{n-2}\sigma_{n}\sigma_{n-1}$. 

\begin{definition}
  Given a permutation $\sigma$, one can sort it as follows. First, start at $\sigma_1$ and compare to $\sigma_{-1} := \sigma_n$. If $\sigma_n < \sigma_{1}$, then swap them. If not, then do nothing. Continue to sort by comparing $\sigma_2$ and $\sigma_1$, then $\sigma_3$ and $\sigma_2$, etc., running through the permutation as many times as needed until the identity permutation is reached. The number of swaps necessary to sort $\sigma$ is the \textbf{number of cyclically simple transpositions needed to sort the permutation}.
\end{definition}

\begin{example}
  Let $\sigma = 2 4 3 1$. We start with $\sigma_1 = 2$. Since $\sigma_1 > \sigma_{-1} = \sigma_4$, we swap them to get $\sigma' = 1 4 3 2$, where the values $1$ and $2$ increase from left to right. Next consider $\sigma'_2 = 4$. Since $\sigma'_2 > \sigma'_1$, we don't switch the values. Then consider $\sigma'_3 = 3$. Since $\sigma'_2 > \sigma'_3$, switch them to get $\sigma'' = 1 3 4 2$. This process continues so next we switch $2$ and $4$ to get $1 3 2 4$. Continuing through the permutation, we don't make another switch until returning to $2$ and $3$, which finishes sorting the permutation. So, the number of cyclically simple transpositions needed to sort $\sigma = 2 4 3 1$ is $4$.
\end{example}

\begin{prop}\label{cyclically} For $n \geq 2,$ the number of cyclically simple transpositions needed to sort a permutation (Statistic $1579$) exhibits the cyclic sieving phenomenon under involutions without fixed points.
\end{prop}
Proposition \ref{cyclically} is proved using $\psi(\sigma) = \sigma_n\sigma_2 \dotsm \sigma_{n-1}\sigma_1.$ 

We now give  statistic definitions and examples related to Proposition ~\ref{transpositions}. Note, these statistics are not equidistributed, but we include them in the same theorem since the same proof method works for both.
\begin{definition}
  Let $\tau_a = (a, a+1)$ for $1 \leq a \leq n$, where $n+1$ is identified by 1. Then \textbf{the minimal length of a factorization of a permutation into transpositions that are cyclic shifts of $(12)$} is given by
  \[\textrm{stat1076}(\sigma) = \textrm{min}\{ k \ | \ \sigma = \tau_{i_1}\tau_{i_2}\dotsm \tau_{i_k} \ 1\leq i_1, i_2, \dotsc i_k \leq n\}.\]
\end{definition}

\begin{example}
  Let $\sigma = 2431$. Then, the minimal length of a factorization of $\sigma$ into transpositions that are cyclic shifts of $(12)$ is $2$ because $2431 = (41)(12)$.
\end{example}

\begin{definition}
  Let $\tau_a = (1, a)$ for $2 \leq a \leq n$. The \textbf{prefix exchange distance of a permutation} $\sigma$ is given by
  \[\textrm{stat1077}(\sigma) = \textrm{min}\{k \ | \ \sigma = \tau_{i_1}, \tau_{i_2} \dotsm \tau_{i_k}, \ 2 \leq i_1, i_2, \dotsc, i_k \leq n\}.\]
\end{definition}

\begin{example} Let $\sigma = 2431$. Then, the prefix exchange distance of $\sigma$ is 2 since $2431 = (14)(12)$.
\end{example}

\begin{prop}\label{transpositions} For $n \geq 2,$ the following statistics exhibit the cyclic sieving phenomenon under involutions without fixed points.
  \begin{itemize}
    \item Statistic $1076$: The minimal length of a factorization of a permutation into transpositions that are cyclic shifts of $(12)$,
    \item Statistic $1077$: The prefix exchange distance.
  \end{itemize}
\end{prop}

Proposition \ref{transpositions} is proved using $\psi(\sigma) = \sigma_2\sigma_1\sigma_3 \dotsm \sigma_n$.

When considering even and odd inversions, we saw that only the number of even inversions exhibited the cyclic sieving phenomenon under involutions with no fixed points. If instead we consider odd and even descents, we see that both statistics do.

\begin{definition}
  A descent is called an \textbf{odd descent} if the index is odd. Similarly, we call it an \textbf{even descent} if the index is even.
\end{definition}

\begin{prop}\label{odd_even_des}
  The following statistics exhibit the cyclic sieving phenomenon under involutions without fixed points.
  \begin{itemize}
    \item For $n \geq 2$, Statistic $1114$: The number of odd descents
    \item For $n \geq 3$, Statistic $1115$: The number of even descents
  \end{itemize}
\end{prop}

The result holds with $n = 2$ for Statistic 1114 as the generating function is $f(q) = 1 + q$, but it does not hold for Statistic 1115 as the generating function is $f(q) = 2$. For $n \geq 3$, Proposition \ref{odd_even_des} is proved using $\psi_o(\sigma) = \sigma_2\sigma_1\sigma_3 \dotsm \sigma_n$ for odd descents and $\psi_e(\sigma) = \sigma_1\sigma_3\sigma_2 \dotsm \sigma_n$ for even descents. 
\smallskip

The following statistic is used in Proposition~\ref{vis_inv}. 
\begin{definition}
  A \textbf{visible inversion} of $\sigma$ is a pair $i < j$ such that $\sigma_j \leq \textrm{min} \{i, \sigma_i\}$.
\end{definition}

\begin{example}
  Let $\sigma = 2431$. Then $(1, 4)$ is a visible inversion of $\sigma$ since $1 < 4$ and $\sigma_4 \leq \textrm{min} \{1, \sigma_1\}$, but $(2, 3)$ is an inversion that is not a visible inversion because $2 < 3$ and $\sigma_3 < \sigma_2$ but $\sigma_3 > \textrm{min} \{2, \sigma_2\}$.
\end{example}

\begin{prop}\label{vis_inv}
  For $n \geq 2$, the number of visible inversions (Statistic $1726$) exhibits the cyclic sieving phenomenon under involutions without fixed points.
\end{prop}

Proposition \ref{vis_inv} is proved with $\psi(\sigma) = \sigma_1\sigma_2 \dotsm \sigma_n\sigma_{n-1}$.
\smallskip

Recall  the definition of permutation patterns from Definition~\ref{def:patterns}.

\begin{thm}\label{conj:patterns}
  For $n \geq 2,$ the number of occurrences of the pattern $231$ or of the pattern $321$ (Statistic $436$) exhibits the cyclic sieving phenomenon under involutions without fixed points.\end{thm}

\begin{proof}
  The number of occurrences of the pattern $231$ or $321$ can be split into two types. In the first type, the pattern involves the value 1, and in the second type, the pattern does not involve the value 1. If $\sigma_i = 1$ then the number of patterns involving 1 is $0$ for $i = 1, 2$ and is counted by $\binom{i-1}{2}$ for $i > 2$, as any pair $\sigma_j, \sigma_k$ with $1 \leq j < k < i$ contributes to one of the patterns.

  We define a map $\psi$ on the set of permutations of $[n]$ that is an involution with no fixed points such that across any orbit $\{\sigma, \psi(\sigma)\}$ the statistic changes in parity.

  First, consider the set of permutations of $[n]$ such that $\sigma_1 = 1$. The number of occurrences of the patterns involving 1 is 0, so we can reduce to counting the number of occurrences of the patterns in $\sigma' = (\sigma_2-1)(\sigma_3-1)\dotsm (\sigma_n-1)$, which is a permutation of $[n-1]$.  By induction, $\sigma'$ and $\psi(\sigma')$ have statistics differing in parity. Define $\psi(\sigma) = \sigma_1 (\psi(\sigma')_1 + 1)\dotsm (\psi(\sigma')_{n-1} + 1)$. As $\sigma_1$ contributes 0 to the statistic, the statistic changes in parity across the orbit $\{\sigma, \psi(\sigma)\}$ .

  Next, consider the set of permutations of $[n]$ with $\sigma_1 \neq 1$. Let $\sigma_i = 1$ and consider $n$ odd. Define $\psi$ on this set of permutations such that $\psi(\sigma)$ has $\sigma_i$ and $\sigma_{i+1}$ switched if $i$ is even and $\sigma_{i}$ and $\sigma_{i - 1}$ switched if $i$ is odd. The only patterns lost or gained by this switch are those involving the value 1 and either $\sigma_{i+1}$ if $i$ is even or $\sigma_{i-1}$ if $i$ is odd. If $i$ is even, $\psi(\sigma)$ gains $i - 1$ patterns contributing to the statistic (one for each $\sigma_j$ with $j < i$). If $i$ is odd, $\psi(\sigma)$ loses $i - 2$ patterns contributing to the statistic (one for each $\sigma_j$ with $j < i - 1$).

  For $n$ even, define $\psi$ as above when restricted to the set of permutations of $[n]$ with $\sigma_1 \neq 1$ and $\sigma_n \neq 1$. These pair as we wish and the only thing remaining to show is the case when $\sigma_n = 1$.

  For all $\sigma$ in the set of permutations of $[n]$ even with $\sigma_n = 1$, 1 contributes to $\binom{n-1}{2}$ patterns. Set $\sigma' =
    (\sigma_1-1)(\sigma_2 -1)\dotsm(\sigma_{n-1}-1)$, which is a permutation of $[n-1]$. By induction, $\sigma'$ and $\psi(\sigma')$ have statistics differing in parity. Define $\psi(\sigma) = \psi(\sigma')_1 + 1)\dotsm (\psi(\sigma')_{n-1} + 1)\sigma_n$. As $\sigma_1$ contributes $\binom{n-1}{2}$ to the statistic for each permutation, the statistic changes in parity across the orbit $\{\sigma, \psi(\sigma)\}$.

  So, for $n \geq 2$, the number of permutations with an even number of occurrences of the pattern 321 or the pattern 231 equals the number of permutations with an odd number of occurrences. Thus, the statistic exhibits the cyclic sieving phenomenon under involutions without fixed points.\end{proof}
  
  \begin{cor}\label{patterns}
  The following statistics are equidistributed with the number of occurrences of the pattern $231$ or of the pattern $321$ (Statistic $436$) and thus also exhibit the cyclic sieving phenomenon under involutions without fixed points.

  \begin{itemize}
    \item Statistic $423$: The number of occurrences of the pattern $123$ or of the pattern $132$
    \item Statistic $428$: The number of occurrences of the pattern $123$ or of the pattern $213$
    \item Statistic $437$: The number of occurrences of the pattern $312$ or of the pattern $321$
  \end{itemize}
\end{cor}

\subsection{Conjecture}

We conclude the section with a conjecture.

\begin{conj}\label{conj:inv_distance_3}
  For $n \geq 2$ even, the number of inversions of distance at most 3  (Statistic $494$) exhibits the cyclic sieving phenomenon under involutions without fixed points.
\end{conj}

Conjecture \ref{conj:inv_distance_3} has been verified for $n\leq 10$.

 
\section{Maps with orbits that all have the same size}
\label{sec:basic}

In this section, we prove some CSPs involving  well-known statistics with maps whose orbits all have the same size. Section \ref{sec:rev_comp} considered a special case of such maps, those for which all the orbits have size 2.
\subsection{Mahonian statistics}
Mahonian statistics are among the most well-known permutation statistics.

\begin{definition}
  A \textbf{Mahonian} statistic $\Stat$ is a statistic on permutations whose generating function is
  \[ \sum_{\sigma \in S_n} q^{\Stat(\sigma)} = (1+q)(1+q+q^2)\ldots(1+q+\ldots+q^{n-1}) =: [n]_q!.\]
  Here, $[n]_q!=[1]_q[2]_q\cdots[n]_q$ where for any $k\in\mathbb{N}$, $[k]_q$ denotes the \emph{$q$-analogue} $[k]_q:=1+q+\ldots+q^{k-1}$.
\end{definition}
The most famous Mahonian statistics are the inversion number (Statistic $18$) and the major index (Statistic $4$). As of January 30, 2024, FindStat contained 19 Mahonian statistics, listed in Appendix \ref{appendix:mahonian}.

\begin{thm}\label{Mahonian}
  Mahonian statistics exhibit the cyclic sieving phenomenon under any  map whose orbits all have size $d$, where $1 \leq d\leq n$.
\end{thm}

\begin{proof}
  The generating function for any Mahonian statistic is given as
  \[
    f(q) = [n]_q! = \prod_{i=1}^n [i]_q =  \prod_{i=1}^n (q^{i-1}+q^{i-2} +\ldots +1).
  \]
  Let $g$ be an action whose orbits all have size $d$ with $1\leq d\leq n$. We consider the primitive $d$th root of unity, $\zeta = e^{\frac{2\pi \cdot i}{d}}$.  We note that $f(\zeta^d)=f(1)=n!$.
  To prove that $g$ exhibits the CSP with respect to this polynomial, we wish to show $f(\zeta^k)=0$ for all $1\leq k< d$.

  By definition, $(\zeta^k)^{d-1}+(\zeta^k)^{d-2} +\cdots + (\zeta^k)^{1} +1=0$ for all $k$ values. Thus, we have
  \begin{align*}
    f(\zeta^k) & =\left (\prod_{i=1}^{d-1}[i]_{\zeta^k} \right) \cdot [d]_{\zeta^k} \cdot \left (\prod_{i=d+1}^{n}[i]_{\zeta^k} \right ) = \left (\prod_{i=1}^{d-1}[i]_{\zeta^k} \right) \cdot 0 \cdot \left( \prod_{i=d+1}^{n}[i]_{\zeta^k} \right)  = 0
  \end{align*}
  as desired.
\end{proof}


\begin{example}
The inversion number statistic under rotation (Map 179) is one of the original instances of the CSP proved in \cite[Theorem 1.1]{ReStWh2004}. Rotation of permutations of $[n]$ has all orbits of size $n$. 
\end{example}

Defant defines {toric promotion} as a bijection on the labelings of a graph. In \cite{ToricPromotion}, he shows that the order of toric promotion has a simple expression if the graph is a forest. Toric promotion can be applied to the permutation $\sigma$ by considering the path labeled from left to right by $\sigma_1\ldots \sigma_n$ as graph, then applying toric promotion to that graph.
Applying toric promotion to that path graph is equivalent to the following definition directly on permutations:
\begin{definition}
\label{def:toric}
Let $\sigma$ be a permutation of $[n]$. Define $\tau_{i,j}(\sigma)= (i,j) \circ \sigma$ if $|\sigma^{-1}(i)-\sigma^{-1}(j)| >1$, and $\tau_{i,j}(\sigma) = \sigma$ otherwise. \textbf{Toric promotion} (Map 310) is equivalent to multiplying the permutation $\sigma$ by the product $\tau_{n,1}\tau_{n-1,n}\ldots\tau_{1,2}$.
\end{definition}

Using \cite[Theorem 1.3]{ToricPromotion}, toric promotion on any tree with $n$ vertices has order $n-1$, and all orbits of toric promotion have size $n-1$. Consequently, toric promotion divides the set of permutations of $[n]$ into orbits all of size $n-1$. So we have the following as a corollary of Theorem \ref{Mahonian}. 

\begin{cor}\label{Mahonian_cor}
   Mahonian statistics exhibit the cyclic sieving phenomenon under the reverse, complement, toric promotion, and rotation maps.
\end{cor}

Mahonian statistics  also appeared in Subsection~\ref{sec:conj}, where we discussed prior work on cyclic sieving of linear combinations of these statistics with respect to conjugation by the long cycle.

\subsection{Rank}
In this subsection, we prove CSPs involving the rank statistic and several maps: reverse, complement, rotation, and Lehmer code rotation.
\begin{definition}
  The \textbf{rank} (Statistic 20) of a permutation of $[n]$ is its position among the $n!$ permutations, ordered lexicographically. This is an integer between $1$ and $n!$.
\end{definition}

\begin{prop}\label{thm:rank}
  The rank (Statistic $20$) of a permutation exhibits the cyclic sieving phenomenon under any bijective action whose orbits have the same size.
\end{prop}
\begin{proof} If all orbits have the same size $d$, $d$ must divide $n!$.
  Every permutation is given a unique rank from 1 to $n!$. So the generating function of the rank statistic is
$
    f(q)= \displaystyle\sum_{j=1}^{n!}q^j
=q\frac{1-q^{n!}}{1-q}$ which vanishes for $q = e^{2\pi ik/d}$ whenever $1\leq k < d$ and $d \ | \ n!$.
\end{proof}

\begin{definition}
  \label{def:lehmercode}
  The \textbf{Lehmer code} of a permutation $\sigma\in S_n$ is
  $L(\sigma )=(L(\sigma )_{1},\ldots, L(\sigma )_{n})\quad {\text{where}}\quad L(\sigma )_{i}=\#\{j>i\mid \sigma _{j}<\sigma _{i}\}$.
  The \textbf{Lehmer code rotation} (Map 149) is a map that sends $\sigma$ to the unique permutation $\tau\in S_n$ such that every entry in the Lehmer code of $\tau$ is cyclically (modulo $n+1-i$) one larger than the Lehmer code of $\sigma$.
\end{definition} 

The Lehmer code rotation has orbits all of size $\lcm(1,2,\ldots,n)$ \cite[Theorem 4.8]{ELMSW22}. Thus, the corollary below follows from Proposition \ref{thm:rank}.
 
\begin{cor}\label{rank_cor}
The rank (Statistic $20$)  exhibits the cyclic sieving phenomenon under the actions of the reverse, complement, toric promotion, rotation, and Lehmer code rotation maps.
\end{cor}

\begin{remark}
  It is interesting that this is the only statistic from FindStat for which the Lehmer code rotation exhibits the CSP, while we showed in \cite{ELMSW22} that this map exhibits many homomesies with statistics in FindStat.
  See Remark \ref{rem:list_of_overlap} for further discussion.
\end{remark}

\subsection{Specific entries and rotation}
In the prior two subsections, we showed the rotation map exhibits the CSP for Mahonian statistics and the rank of a permutation. Here we prove an additional CSP for the rotation map with any specific entry statistic. This includes the following FindStat statistics:
the first entry of a permutation (Statistic 54), as well as its last (Statistic 740), the upper middle (Statistic 1806) and lower middle (Statistic 1807) entries.

\begin{prop}\label{prop:ith_entry}
  The $i$-th entry of a permutation exhibits the cyclic sieving phenomenon under any action whose orbits all have size $d$ where $d \ | \ n$, in particular, the rotation map. 
\end{prop}

\begin{proof}
  The generating function for the $i$-th entry of a permutation is $f(q) = \displaystyle\sum_{j = 1}^n (n-1)! q^j=(n-1)!q\frac{1-q^n}{1-q}$. Note $\frac{1-q^n}{1-q}$ vanishes for $q=e^{2\pi i k/d}$ whenever $1\leq k < d$ and $d \ | \ n$.
%
\end{proof}

\subsection{Inversions of a specific entry and toric promotion}
In the prior two subsections, we showed the toric promotion map exhibits the CSP for Mahonian statistics and the rank of a permutation. Here we prove two additionals CSP for the toric promotion map.

The number of inversions of the $i$-th entry of a permutation $\sigma$ is the number of inversions $\{ i < j \mid \sigma(j) < \sigma(i)\}$, where $i$ is fixed. In FindStat, we find the following statistics: the number of inversions of the second entry (Statistic 1557) and of the third entry (Statistic 1556), as well as the first entry of a permutation (Statistic 54). Note that the value of the first entry of a permutation is one more than the number of inversions of the first entry.

\begin{thm}\label{thm:ith_entry}
  The number of inversions of the $i$-th entry exhibits the cyclic sieving phenomenon under any map whose orbits all have size $d$ where $d \ | \ (n+1-i)$.
\end{thm}
\begin{proof}
    We use the fact that permutations of $[n]$ are in bijection with Lehmer codes, which are sequences of $n$ numbers, such that the $i$-th takes a value between $0$ and $n-i$. There are $n!$ possible Lehmer codes, and the Lehmer code $(a_1, \ldots, a_n)$ corresponds to the unique permutation that has $a_i$ inversions of the $i$-th entry for each $1 \leq i \leq n$. For more details on this bijection with Lehmer codes, see for example \cite[p.12]{Knuth_AOCP3}.
    Since the number of Lehmer codes with $i$-th entry equal to $k$ is $\frac{n!}{n-i+1}$ for any value of $k$ between $0$ and $n-i$, there are exactly $\frac{n!}{n-i+1}$ permtutations with $k$ inversions of the $i$-th entry, and the statistic generating function is
    \[ f(q) = \frac{n!}{n-i+1}\sum_{j = 0}^{n-i} q^j=\frac{n!}{n-i+1}\frac{1-q^{n+1-i}}{1-q} \]
    which vanishes for $q=e^{2\pi i k/(n+1-i)}$ when $1\leq k < d$ and $d \ | \ (n+1-i)$. 
\end{proof}

Note that the above result allows us to recover the result of Proposition \ref{thm:ith_entry} for the first entry, a CSP for the reverse and complement maps with the number of inversions of the $(n-1)$st entry, and the following corollary about toric promotion (see Definition~\ref{def:toric}).
\begin{cor}\label{cor:toric_pro_2nd_entry}
The number of inversions of the second entry (Statistic $1557$) exhibits the  cyclic sieving phenomenon under the toric promotion map.
\end{cor}

There is one more statistic that FindStat suggests as exhibiting the CSP for toric promotion. 
\begin{definition}
For $\sigma\in S_n$, define the following \textbf{descent variant} $\partial(\sigma)=\sum_{i\in \Des(\sigma)}{i(n-i)}$.
\end{definition}

\begin{prop}\label{thm:1911}
The descent variant $\partial$ minus the number of inversions (Statistic 1911) exhibits the cyclic sieving phenomenon under the toric promotion map.
\end{prop}
\begin{proof}
By \cite[Remark 1.5]{StembridgeWaugh}, the generating function for the descent variant statistic is 
\[\sum_{\sigma\in S_n}q^{\partial(\sigma)-\inv(\sigma)}=n\prod_{i=1}^{n-1}\frac{1-q^{i(n-1)}}{1-q^i}.
\]
The factor of this product when $i=1$ is $\displaystyle\frac{1-q^{n-1}}{1-q}$, which equals $0$ whenever $q=e^{2\pi ik/(n-1)}$ for $1\leq k < n-1$. 
\end{proof}

\bibliographystyle{plain}
\bibliography{master.bib}

\begin{thebibliography}{10}

\bibitem{CSPwebsite}
Per Alexandersson.
\newblock The symmetric functions catalog.
\newblock Online.
\newblock \url{https://www.symmetricfunctions.com/cyclic-sieving.htm}.

\bibitem{AlexAmini}
Per Alexandersson and Nima Amini.
\newblock The cone of cyclic sieving phenomena.
\newblock {\em Discrete Math.}, 342(6):1581--1601, 2019.

\bibitem{AlexanderssonKebede}
Per Alexandersson and Frether Getachew.
\newblock An involution on derangements preserving excedances and right-to-left minima.
\newblock {\em Australas. J. Combin.}, 86(3):387--413, 2023.

\bibitem{BabsonSteingrimsson}
Eric Babson and Einar Steingr\'{\i}msson.
\newblock Generalized permutation patterns and a classification of the {M}ahonian statistics.
\newblock {\em S\'{e}m. Lothar. Combin.}, 44:Art. B44b, 18 pages, 2000.

\bibitem{BRS}
H\'{e}l\`ene Barcelo, Victor Reiner, and Dennis Stanton.
\newblock Bimahonian distributions.
\newblock {\em J. Lond. Math. Soc. (2)}, 77(3):627--646, 2008.

\bibitem{TennerArrow}
Yosef Berman and Bridget~E. Tenner.
\newblock Pattern-functions, statistics, and shallow permutations.
\newblock {\em Electron. J. Combin.}, 29(4):Paper No. 4.43, 20, 2022.

\bibitem{fingerprint}
Sara~C. Billey and Bridget~E. Tenner.
\newblock Fingerprint databases for theorems.
\newblock {\em Notices Amer. Math. Soc.}, 60(8):1034--1039, 2013.

\bibitem{ClarkeSteingrimssonZeng}
Robert Clarke, Einar Steingr\'{\i}msson, and Jiang Zeng.
\newblock New {E}uler-{M}ahonian statistics on permutations and words.
\newblock {\em Adv. Appl. Math.}, 18:237–270, 1997.

\bibitem{corteel2007crossings}
Sylvie Corteel.
\newblock Crossings and alignments of permutations.
\newblock {\em Adv. in Appl. Math.}, 38(2):149--163, 2007.

\bibitem{widthk}
Robert Davis.
\newblock Width-k generalizations of classical permutation statistics.
\newblock {\em J. Integer Seq.}, Vol. 20(17.6.3), 2017.

\bibitem{BishalDeb}
Bishal Deb.
\newblock Continued fractions using a {L}aguerre digraph interpretation of the {F}oata--{Z}eilberger bijection and its variants, 2023.
\newblock \url{arXiv:2304.14487}.

\bibitem{ToricPromotion}
Colin Defant.
\newblock Toric promotion.
\newblock {\em Proc. Amer. Math. Soc.}, 151(1):45--57, 2023.

\bibitem{dowling2023homomesy}
William Dowling and Nadia Lafreni\`ere.
\newblock Homomesy on permutations with toggling actions, 2023.
\newblock To appear in \textit{Involve}, \href{https://arxiv.org/abs/2312.02383}{ArXiv:2312.02383}.

\bibitem{ELMSW22}
Jennifer Elder, Nadia Lafreni{\`e}re, Erin McNicholas, Jessica Striker, and Amanda Welch.
\newblock Homomesies on permutations: An analysis of maps and statistics in the {F}ind{S}tat database.
\newblock {\em Math. Comp.}, 2023.

\bibitem{elizalde2018continued}
Sergi Elizalde.
\newblock Continued fractions for permutation statistics.
\newblock {\em Discrete Math. Theor. Comput. Sci.}, 19(2):Paper No. 11, 24, 2017.

\bibitem{FoataHan}
Dominique Foata and Guo-Niu Han.
\newblock Fix-{M}ahonian calculus. {II}. {F}urther statistics.
\newblock {\em J. Combin. Theory Ser. A}, 115(5):726--736, 2008.

\bibitem{FoataZeilberger}
Dominique Foata and Doron Zeilberger.
\newblock Denert's permutation statistic is indeed {E}uler-{M}ahonian.
\newblock {\em Stud. Appl. Math.}, 83(1):31--59, 1990.

\bibitem{FranconViennot1979}
Jean Fran\c~con and G\'erard Viennot.
\newblock Permutations selon leurs pics, creux, doubles mont\'ees et double descentes, nombres d'{E}uler et nombres de {G}enocchi.
\newblock {\em Discrete Math.}, 28(1):21--35, 1979.

\bibitem{Knuth_AOCP3}
Donald~E. Knuth.
\newblock {\em The art of computer programming. {V}olume 3: Sorting and searching}.
\newblock Addison-Wesley Publishing Company, 1973.

\bibitem{LinZeng}
Zhicong Lin and Jiang Zeng.
\newblock The $\gamma$-positivity of basic {E}ulerian polynomials via group actions.
\newblock {\em J. Comb. Theory Ser. A}, 135:112--129, 2015.

\bibitem{CycleDescent}
Jun Ma, Shi-Mei Ma, Yeong-Nan Yeh, and Xu~Zhu.
\newblock The cycle descent statistic on permutations.
\newblock {\em Electron. J. Combin.}, 23(4):Paper 4.20, 24, 2016.

\bibitem{MacMahon}
Percy~A. MacMahon.
\newblock {\em Combinatory analysis. {V}ol. {I}, {II} (bound in one volume)}.
\newblock Dover Phoenix Editions. Dover Publications, Inc., Mineola, NY, 2004.
\newblock Reprint of {\it An introduction to combinatory analysis} (1920) and {\it Combinatory analysis. Vol. I, II} (1915, 1916).

\bibitem{Petersensortingindex}
T.~Kyle Petersen.
\newblock The sorting index.
\newblock {\em Adv. in Appl. Math.}, 47(3):615--630, 2011.

\bibitem{reading2015noncrossing}
Nathan Reading.
\newblock Noncrossing arc diagrams and canonical join representations.
\newblock {\em SIAM J. Discrete Math.}, 29(2):736--750, 2015.

\bibitem{ReStWh2004}
Victor Reiner, Dennis Stanton, and Dennis White.
\newblock The cyclic sieving phenomenon.
\newblock {\em J. Combin. Theory Ser. A}, 108(1):17 -- 50, 2004.

\bibitem{WhatIsCSP}
Victor Reiner, Dennis Stanton, and Dennis White.
\newblock What is \dots cyclic sieving?
\newblock {\em Notices Amer. Math. Soc.}, 61(2):169--171, 2014.

\bibitem{Rodrigues}
Olinde Rodrigues.
\newblock Note sur les inversions, ou dérangements produits dans les permutations.
\newblock {\em J. Math. Pures Appl.}, 4:236--240, 1839.

\bibitem{findstat}
Martin Rubey, Christian Stump, et~al.
\newblock {FindStat} - {T}he combinatorial statistics database.
\newblock \url{http://www.FindStat.org}.
\newblock Accessed: April 27, 2022.

\bibitem{Sagan1987}
Bruce~E. Sagan.
\newblock Shifted tableaux, {S}chur {$Q$}-functions, and a conjecture of {R}. {S}tanley.
\newblock {\em J. Combin. Theory Ser. A}, 45(1):62--103, 1987.

\bibitem{SaganCSP}
Bruce~E. Sagan.
\newblock The cyclic sieving phenomenon: a survey.
\newblock In {\em Surveys in combinatorics 2011}, volume 392 of {\em London Math. Soc. Lecture Note Ser.}, pages 183--233. Cambridge Univ. Press, Cambridge, 2011.

\bibitem{SSW}
Bruce~E. Sagan, John Shareshian, and Michelle~L. Wachs.
\newblock Eulerian quasisymmetric functions and cyclic sieving.
\newblock {\em Adv. in Appl. Math.}, 46(1-4):536--562, 2011.

\bibitem{Schur1911}
J.~Schur.
\newblock \"{U}ber die {D}arstellung der symmetrischen und der alternierenden {G}ruppe durch gebrochene lineare {S}ubstitutionen.
\newblock {\em J. Reine Angew. Math.}, 139:155--250, 1911.

\bibitem{ShareshianWachs2007}
John Shareshian and Michelle~L. Wachs.
\newblock {$q$}-{E}ulerian polynomials: excedance number and major index.
\newblock {\em Electron. Res. Announc. Amer. Math. Soc.}, 13:33--45, 2007.

\bibitem{ShareshianWachs2016}
John Shareshian and Michelle~L. Wachs.
\newblock Chromatic quasisymmetric functions.
\newblock {\em Adv. Math.}, 295:497--551, 2016.

\bibitem{SimionSchmidt}
Rodica Simion and Frank~W. Schmidt.
\newblock Restricted permutations.
\newblock {\em Eur. J. Comb.}, 6:383--406, 1985.

\bibitem{Stanley2011}
Richard~P. Stanley.
\newblock {\em Enumerative Combinatorics: Volume 1}.
\newblock Cambridge University Press, New York, NY, USA, 2nd edition, 2011.

\bibitem{sage}
William\thinspace{}A. Stein et~al.
\newblock {\em {S}age {M}athematics {S}oftware ({V}ersion 10.0)}.
\newblock The Sage Development Team, 2023.
\newblock \url{http://www.sagemath.org}.

\bibitem{SteingrimssonWilliams2007}
Einar Steingr\'{\i}msson and Lauren~K. Williams.
\newblock Permutation tableaux and permutation patterns.
\newblock {\em J. Combin. Theory Ser. A}, 114(2):211--234, 2007.

\bibitem{JStembridge}
John Stembridge.
\newblock Some hidden relations involving the ten symmetry classes of plane partitions.
\newblock {\em J. Combin. Theory Ser. A.}, 68(2):372--409, 1994.

\bibitem{Stembridge1994}
John~R. Stembridge.
\newblock On minuscule representations, plane partitions and involutions in complex {L}ie groups.
\newblock {\em Duke Math. J.}, 73(2):469--490, 1994.

\bibitem{StembridgeWaugh}
John~R. Stembridge and Debra~J. Waugh.
\newblock A {W}eyl group generating function that ought to be better known.
\newblock {\em Indag. Math. (N.S.)}, 9(3):451--457, 1998.

\bibitem{ThamrongpairojRemmel}
Sittipong Thamrongpairoj and Jeffrey~B. Remmel.
\newblock Positional marked patterns in permutations.
\newblock {\em Discrete Math. Theor. Comput. Sci.}, 24(1):Paper No. 23, 27, 2022.

\bibitem{Viennot_heaps}
G\'{e}rard~Xavier Viennot.
\newblock Heaps of pieces. {I}. {B}asic definitions and combinatorial lemmas.
\newblock In {\em Graph theory and its applications: {E}ast and {W}est ({J}inan, 1986)}, volume 576 of {\em Ann. NY Acad. Sci.}, pages 542--570. New York Acad. Sci., New York, 1989.

\bibitem{LW2005}
Lauren~K. Williams.
\newblock Enumeration of totally positive {G}rassmann cells.
\newblock {\em Adv. Math.}, 190(2):319--342, 2005.

\end{thebibliography}

\appendix
\section{List of Mahonian statistics}\label{appendix:mahonian}
\begin{prop}
  The following FindStat statistics on permutations are Mahonian:
  \begin{itemize}
    \item Statistic $4$: The major index
    \item Statistic $18$: The number of inversions
    \item Statistic $156$: The Denert index
    \item Statistic $224$: The sorting index
    \item Statistic $246$: The number of non-inversions
    \item Statistic $304$: The load
    \item Statistic $305$: The inverse major index
    \item Statistic $334$: The maz index, the major index after replacing fixed points by zeros
    \item Statistic $339$: The maf index
    \item Statistic $446$: The disorder
    \item Statistic $692$: Babson and Steingr\'imsson's statistic stat
    \item Statistic $794$: The mak
    \item Statistic $795$: The mad
    \item Statistic $796$: The stat$'$
    \item Statistic $797$: The stat$''$
    \item Statistic $798$: The makl
    \item Statistic $833$: The comajor index
    \item Statistic $868$: The aid statistic in the sense of Shareshian-Wachs
    \item Statistic $1671$: Haglund's hag
  \end{itemize}
\end{prop}

\begin{proof}
  All these statistics are known in the literature to be Mahonian or can be seen to be easily equidistributed with Mahonian statistics \cite{BabsonSteingrimsson,FoataHan,FoataZeilberger,MacMahon,Petersensortingindex,Rodrigues,ShareshianWachs2007,Stanley2011}.
\end{proof}

\end{document}